\documentclass[12pt]{amsart}
\usepackage{amsfonts}
\usepackage{graphicx}
\usepackage{tabularx}
\usepackage{array}
\usepackage[usenames,dvipsnames]{color}
\usepackage{comment}
\usepackage{amsmath}
\usepackage{amsthm}
\usepackage{amssymb}
\usepackage{fullpage}
\usepackage[dvipsnames]{xcolor}
\usepackage{tikz}
\usepackage{listings}
\usepackage{fontenc}

\DeclareMathOperator{\dist}{dist}
\DeclareMathOperator{\mad}{mad}
\DeclareMathOperator{\ch}{ch}
\DeclareMathOperator{\spl}{split}

\newtheorem{theorem}{Theorem}[section]

\newtheorem{observation}[theorem]{Observation}

\newtheorem{claim}[theorem]{Claim}
\newtheorem{lemma}[theorem]{Lemma}

\theoremstyle{definition}
\newtheorem{definition}[theorem]{Definition}

\def\epsilon{\varepsilon}

\title{Graphs of maximum average degree less than $\frac {11}{3}$ are flexibly $4$-choosable}
\author{Richard Bi}
\address{Department of Mathematics, University of Illinois Urbana-Champaign}
\email{rbi3@illinois.edu}

\author{Peter Bradshaw}
\address{Department of Mathematics, University of Illinois Urbana-Champaign}
\email{pb38@illinois.edu}
\thanks{Peter Bradshaw received funding from NSF RTG grant DMS-1937241.}

\begin{document}
\maketitle
\begin{abstract}
We consider the \emph{flexible list coloring} problem, in which we have a graph $G$, a color list assignment $L:V(G) \rightarrow 2^{\mathbb N}$, and a set $U \subseteq V(G)$ of vertices such that each $u \in U$ has a preferred color $p(u) \in L(u)$. Given a constant $\epsilon > 0$, the problem asks for an $L$-coloring of $G$ in which at least $\epsilon |U|$ vertices in $U$ receive their preferred color. 
We use a method of reducible subgraphs to approach this problem.
We develop a vertex-partitioning tool that, when used with a new reducible subgraph framework, allows us to define large reducible subgraphs. Using this new tool, we show that if $G$ has maximum average degree less than $\frac{11}{3}$, a list $L(v)$ of size $4$ at each $v \in V(G)$, and a set $U \subseteq V(G)$ of vertices with preferred colors, then there exists an $L$-coloring of $G$ for which at least $2^{-145} |U|$
vertices of $U$ receive their preferred color.
\end{abstract}
\section{Introduction}
\subsection{Background}
A proper coloring of a graph $G$ is a mapping $\phi:V(G) \rightarrow \mathbb N$ such that for each edge $uv \in E(G)$, $\phi(u) \neq \phi(v)$. A \emph{list assignment} on $G$ is a mapping $L:V(G) \rightarrow 2^{\mathbb N}$ which assigns a set $L(v) \subseteq \mathbb N$ to each vertex $v \in V(G)$. 
Given
a function $f:V(G) \rightarrow \mathbb N$, a list assignment $L:V(G) \rightarrow 2^{\mathbb N}$ is an \emph{$f$-assignment} if $|L(v)| = f(v)$ for each vertex $v \in V(G)$. If $f(v) = k$ for each vertex $v \in V(G)$ and $L$ is an $f$-assignment, then  $L$ is a \emph{$k$-assignment}.
Given a list assignment $L$ on $G$,
a proper coloring $\phi$ of $G$ is called an \emph{$L$-coloring} if $\phi(v) \in L(v)$ for each vertex $v \in V(G)$. The problem of determining whether a graph $G$ has an $L$-coloring for some list assignment $L$ is called the \emph{list coloring problem}.

The list coloring problem was first introduced by Vizing \cite{Vizing} and independently by Erd\H{o}s, Rubin, and Taylor. Vizing \cite{Vizing} showed that a list coloring of a complete graph can be equivalently described as a system of distinct representatives of a set family, and Erd\H{o}s, Rubin, and Taylor \cite{ERT} observed that the list coloring problem for complete bipartite graphs is closely related to the proper $2$-coloring problem for regular hypergraphs.
They also defined the notion of a graph's \emph{choosability} as follows. Given a graph $G$, $G$ is \emph{$k$-choosable} if $G$ has an $L$-coloring for every $k$-assignment $L:V(G) \rightarrow  2^{\mathbb N}$. Then, the choosability of $G$ is the minimum integer $k$ for which $G$ is $k$-choosable. 
We write $\ch(G)$ for the choosability of $G$.
If $L(v) = \{1,\dots,k\}$ for each vertex $v \in V(G)$, then the question of whether $G$ is $L$-colorable is equivalent to the question of whether $G$ is properly $k$-colorable. Therefore, $\ch(G) \geq \chi(G)$.

Given a $k$-choosable graph $G$ and a $k$-assignment $L$ on $G$, 
it is natural to investigate 
the diversity of
the set of $L$-colorings of $G$.
When $G$ is planar, 
one common question asks for the
number of $L$-colorings of $G$.
For each $5$-assignment $L$ on a planar graph $G$,
Thomassen \cite{Thomassen5} proved 
that $G$ has at least one $L$-coloring, and he later proved \cite{ThomassenExp5} that $G$ has at least $2^{|V(G)|/9}$ $L$-colorings.
Using the polynomial method, Bosek et al.~\cite{Bosek} proved further that $G$ in fact has at least $5^{|V(G)|/4}$ $L$-colorings. Similarly, 
for each $3$-assignment $L$ on a planar graph $G$ of girth at least $5$,
Thomassen \cite{Thomassen3} proved that $G$ has at least one $L$-coloring, and he lated proved \cite{Thomassen3Exp} that $G$ in fact has at least $2^{|V(G)| / 10000}$ $L$-colorings.
Again using the polynomial method,
Bosek et al.~\cite{Bosek} showed more strongly that $G$ has at least $3^{|V(G)|/6}$ $L$-colorings. Recently, Postle and Roberge \cite{PR} used a hyperbolicity method 
to obtain similar results in the more general setting of correspondence colorings.

Given a graph $G$ and a list assignment $L$,
another way of investigating the diversity of $L$-colorings of $G$ is through the \emph{flexible list coloring} problem of Dvo\v{r}\'ak, Norin, and Postle \cite{DNP}, 
in which
one assigns a preferred color $p(v) \in L(v)$ to some of the vertices $v \in V(G)$ and then asks
whether $G$ has an $L$-coloring in which many vertices $v$ with a color preference receive their preferred color $p(v)$.
Formally, a
\emph{weighted request}
on a graph $G$
with a list assignment $L$
is a function $w$ such that for each vertex $v \in V(G)$ and color $c \in L(v)$, $w$ maps the pair $(v,c)$ to a nonnegative real number $w(v,c)$.
Given a value $\epsilon > 0$, we say that $G$ is \emph{weighted $\epsilon$-flexibly
$k$-choosable} if for every $k$-assignment $L$ and weighted request $w$ on $G$, there exists an $L$-coloring $\phi$ of $G$ such that 
\begin{eqnarray}
\label{eqn:epsilon}
\sum_{v \in V(G)} w(v,\phi(v)) \geq \epsilon \sum_{v \in V(G)} \sum_{c \in L(v)} w(v,c).
\end{eqnarray}
In other words, the weight of the pairs $(v,c)$ for which $\phi(v) = c$ is at least an $\epsilon$ proportion of the weight of all pairs $(v,c)$.

We say that $w$ is an \emph{(unweighted) request} if for each $v \in V(G)$, $w(v,c) = 1$ for at most one color $c \in L(v)$, and $w(v,c') = 0$ for all other colors $c' \in L(v)$.
Given a graph $G$, if there exists a value $\epsilon > 0$ such that the inequality (\ref{eqn:epsilon}) holds for every $k$-assignment $L$ on $G$ and every unweighted request $w$ on $G$, then we say that $G$ is \emph{$\epsilon$-flexibly $k$-choosable}.

One closely related question to the flexible list coloring problem asks the following: Given a graph $G$, what is the smallest value $k$ for which $G$ is weighted $\frac 1k$-flexibly $k$-choosable? This value $k$ is equivalent to the \emph{fractional list packing number}
introduced by Cambie, Cames van Batenburg, Davies, and Kang
\cite{CCDK}
and later studied by Cambie, Cames van Batenburg, and Zhu \cite{CCZ}.
Cambie, Cames van Batenburg,
Davies, and Kang \cite{CCDK}
showed that if $G$ has maximum degree $\Delta$, then $G$ is weighted $\frac{1}{\Delta+1}$-flexibly $(\Delta+1)$-choosable, implying that the fractional list packing number of every graph is well defined.
Cambie, Cames van Batenburg, and Zhu
\cite{CCZ}
further showed that if $G$ is planar, then $G$ is weighted $\frac 15$-flexibly $5$-choosable when $G$ is triangle-free, weighted $\frac 14$-flexibly $4$-choosable when $G$ has girth at least $5$, and weighted $\frac 13$-flexibly $3$-choosable when $G$ has girth at least $6$.

A graph $G$ is \emph{$d$-degenerate} if every nonempty subgraph of $G$ has a vertex of degree at most $d$, and a greedy argument shows that a $d$-degenerate graph is $(d+1)$-choosable.
Dvo\v{r}\'ak, Norin, and Postle \cite{DNP} asked whether every $d$-degenerate graph is $\epsilon$-flexibly $(d+1)$-choosable for some constant $\epsilon = \epsilon(d) > 0$, and they showed that every $d$-degenerate graph is $\frac{1}{2d^{2d}}$-flexibly $(d+2)$-choosable. This question is open for all values $d \geq 2$.

While proving the flexibly $(d+1)$-choosability of all $d$-degenerate graphs seems difficult, even for a fixed value $d \geq 2$, many existing results show that a large class $\mathcal G$ of $d$-degenerate graphs is $\epsilon$-flexibly $(d+1)$-choosable for some universal constant $\epsilon > 0$ \cite{BMS}. For example, the class of non-regular graphs of maximum degree $\Delta \geq 3$ is a $(\Delta-1)$-degenerate class which is weighted $\frac{1}{2\Delta^4}$-flexibly $\Delta$-choosable. Additionally, the class of triangle-free planar graphs is $3$-degenerate and weighted $2^{-126}$-flexibly $4$-choosable \cite{DMMP-triangle-free}. 

The degeneracy of a graph $G$ is closely related to 
the \emph{maximum average degree} of $G$, written $\mad(G)$, 
defined as 
the maximum value $\frac{2|E(H)|}{|V(H)|}$, where the maximum is taken over all nonempty subgraphs $H \subseteq G$.
In particular, if $G$ has maximum average degree $\overline{d}$, then every nonempty subgraph of $G$ has average degree at most $\overline{d}$ and hence has a vertex of degree at most $\lfloor \overline{d} \rfloor$. Therefore, the degeneracy of $G$ is at most $\lfloor \mad(G) \rfloor$. Hence, for each integer $d \geq 1$ and constant $c \in (0,1]$, the class of graphs of maximum average degree less than $d + c$ is a $d$-degenerate graph class.
The following theorem of Dvo\v{r}\'ak, Norin, and Postle shows that when $c = \frac{2}{d+3}$, the associated $d$-degenerate graph class is $\epsilon$-flexibly $(d+1)$-choosable for some $\epsilon = \epsilon(d) > 0$.

\begin{theorem}[\cite{DNP}] 
\label{thm:2d3}
For each $d \geq 2$, if $G$ is a graph with maximum average degree less than $d + \frac{2}{d+3}$, then $G$ is weighted $\epsilon$-flexibly $(d+1)$-choosable for some constant $\epsilon = \epsilon(d) > 0$. 
\end{theorem}

The proof of Theorem \ref{thm:2d3} from \cite{DNP} is simple and does not attempt
to maximize the quantity $d + \frac{2}{d+3}$, and this quantity is almost certainly not best possible for any $d \geq 2$. 
Hence, it is natural to ask for the maximum value $c \in (0,1]$ such that every graph with maximum average degree less than $d + c$ is $\epsilon$-flexibly $(d+1)$-choosable for some $\epsilon = \epsilon(d) > 0$. 
The authors \cite{RBPB} recently showed that the answer to this question for $d = 2$ is $c=1$, proving that
 every graph with maximum average degree less than $3$ is $2^{-32}$-flexibly $3$-choosable. As $K_4$ has maximum average degree exactly $3$ and is not $3$-choosable, this upper bound on maximum average degree is best possible.

\subsection{Our results}
When $d = 3$, Theorem \ref{thm:2d3} implies that a graph of maximum average degree less than $\frac{10}{3}$ is weighted $\epsilon$-flexibly $4$-choosable for some $\epsilon > 0$. 
The main result of this paper improves this maximum average degree bound.

\begin{theorem}
    \label{thm:mad113}
    There exists a constant $\epsilon > 0$ for which the following holds.
    If $G$ is a graph with maximum average degree less than $\frac{11}{3}$, then $G$ is weighted $\epsilon$-flexibly $4$-choosable.
\end{theorem}

In particular, we show that Theorem \ref{thm:mad113} holds with a constant $\epsilon = 2^{-145}$. We make no effort to optimize $\epsilon$.
To prove Theorem \ref{thm:mad113}, we consider a graph $G$ of maximum average degree less than $\frac{11}{3}$, along with a $4$-assignment $L$ on $G$. Then, we aim to construct a distribution on $L$-colorings of $G$ such that for each $v \in V(G)$ and $c \in L(v)$, $v$ receives the color $c$ with probability at least $\epsilon$.
The following Lemma of Dvo\v{r}\'ak, Norin, and Postle shows that the existence of such a distribution on $L$-colorings is sufficient to prove the weighted $\epsilon$-flexible $4$-choosability of $G$.

\begin{lemma}[\cite{DNP}] \label{lem:prob} 
    Let $G$ be a graph, and let $k$ be a positive integer. Suppose that for every $k$-assignment $L$ on $G$, there exists a probability distribution on $L$-colorings $\phi$ of $G$ such that for each vertex $v \in V(G)$ and color $c \in L(v)$, 
    \[\Pr(\phi(v) = c) \geq \epsilon.\]
    Then, 
    $G$ is weighted $\epsilon$-flexibly $k$-choosable.
\end{lemma}

Given 
a graph $G$ and a $4$-assignment $L$ on $G$, we 
aim to construct a distribution on $L$-colorings of $G$ using a
framework of reducible subgraphs. In a graph coloring problem, an induced subgraph $H$ of a graph $G$ is often called \emph{reducible} if a solution to the coloring problem on $G \setminus H$ also guarantees a solution to the coloring problem on $G$. For example,
when attempting to find an $L$-coloring of a graph $G$,
an induced subgraph $H \subseteq G$ may be called reducible if every $L$-coloring of $G \setminus H$ extends to an $L$-coloring of $G$. In our flexible list coloring setting, we roughly say that $H$ is reducible if every suitable distribution on $L$-colorings of $G \setminus H$ extends to a suitable distribution on $L$-colorings of $G$.
To prove Theorem \ref{thm:mad113}, we show that a minimum counterexample $G$ to the theorem contains no reducible subgraph and therefore satisfies certain structural properties. After these structural properties are established, a short discharging argument shows that $\mad(G) \geq \frac{11}{3}$, contradicting the initial assumption that $G$ violates Theorem \ref{thm:mad113}.

The specific reducibility framework that we use
in this paper is a special case of the reducibility framework introduced by the authors
\cite{RBPB}. 
The framework that we use here and in \cite{RBPB} 
generalizes a previous reducibility framework of Dvo\v{r}\'ak, Masa\v{r}\'ik, Mus\'ilek, and Pangr\'ac \cite{DMMP-triangle-free}.
The main difference between our current reducibility framework and that of \cite{DMMP-triangle-free} is that by 
adding stronger requirements for a subgraph to be reducible, our framework allows for arbitrarily large reducible subgraphs, whereas the framework of \cite{DMMP-triangle-free} only allows reducible subgraphs of bounded size. 

One additional contribution of this paper is a lemma (Lemma \ref{lem:orig_partition}) that facilitates the identification of large reducible subgraphs. The lemma essentially states
that if  an induced subgraph $H$ of $G$
can be partitioned into small parts satisfying certain properties, then $H$ is a reducible subgraph.
This lemma allows us 
to 
identify families of reducible subgraphs of unbounded size based on their global properties, which ultimately
gives us a simple discharging argument.
Furthermore, the lemma allows us to avoid computer-assisted proofs and to use ideas that one can check by hand.

\subsection{Outline of paper}

The paper is structured as follows. In Section \ref{sec:framework}, we define the reducible subgraph framework that we use throughout the paper, and we establish some lemmas which will help us identify reducible subgraphs. In Section \ref{sec:tools}, we establish a partition lemma (Lemma \ref{lem:orig_partition}), which allows us to show that large subgraphs are reducible whenever their vertices can be partitioned into small parts satisfying certain properties. 
The proofs in Section \ref{sec:tools} 
are routine but often tedious, and the impatient reader may skip them without missing any main ideas.
In Section \ref{sec:reducible}, we consider a minimum counterexample $G$ to Theorem \ref{thm:mad113}, and with the help of our partition lemma, we establish strong structural properties of $G$. Finally, in Section \ref{sec:discharging}, we use these structural properties of $G$ to carry out a simple discharging argument that proves that $\mad(G) \geq \frac{11}{3}$. This result implies that $G$ in fact is not a counterexample to Theorem \ref{thm:mad113}, completing the proof.

\subsection{Notation}
Let $G$ be a graph.
If $v \in V(G)$ is a vertex of degree $d$, then we say that $v$ is a \emph{$d$-vertex}. If $\deg(v) \geq d$, then we say that $v$ is a \emph{$d^+$-vertex}.
 If $\deg(v) \leq d$, then we say that $v$ is a \emph{$d^-$-vertex}.
If $v$ is a $d$- ($d^+$-, $d^-$-)  vertex which is adjacent to another vertex $u$, then we often say that $v$ is a \emph{$d$- ($d^+$-, $d^-$-) neighbor} of $u$.
If $P$ is a path for which each $v \in V(P)$ is a $d$- ($d^+$-, $d^-$-)  vertex, then we say that $P$ is a \emph{$d$- ($d^+$-, $d^-$-) path}.

Given a function $f:V(G) \rightarrow \mathbb Z$, we say that an \emph{$f$-assignment} on $G$ is a function $L:V(G) \rightarrow 2^{\mathbb N}$ for which $|L(v)| = \max\{0,f(v)\}$ for each $v \in V(G)$. (When $f$ is positive-valued, this definition agrees with the original definition of an $f$-assignment.) Note that if $f(v) \leq 0$ for some $v \in V(G)$, then $G$ is not $f$-choosable.

Given a vertex $v \in V(G)$ and a vertex subset $X \subseteq V(G)$, we write $N_X(v) = N(v) \cap X$, and we write $N_X[v] = \{v\} \cup N_X(v)$. Given vertex subsets $S \subseteq V(G)$ and $X \subseteq V(G)$, we write $N_X[S] = \bigcup_{v \in S} N_X[v]$.
Given a subgraph $H \subseteq G$ and a vertex $v \in V(G)$, we write $N_H(v) = N_{V(H)}(v)$. We define $N_H[v]$ and  $N_H[S]$ similarly.

If $G$ is a graph with a list assignment $L$, $H$ is an induced subgraph of $G$, and $\phi$ is an $L$-coloring of $G\setminus H$, then we say that a color $c \in L(v)$ is \emph{available} at a vertex $v \in V(H)$ if $\phi(w) \neq c$ for each neighbor $w \in N(v) \setminus V(H)$.

Given two vertices $u,v \in V(G)$, the \emph{distance} between $u$ and $v$ is the number of edges in a shortest path joining $u$ and $v$. If $u$ and $v$ belong to distinct components of $G$, then we say that the distance between $u$ and $v$ is $\infty$.

Given a vertex $v \in V(G)$, we identify $v$ with the $K_1$ subgraph of $G$ induced by $v$; that is, $v = G[\{v\}]$. In other words, we consider $v$ both as a vertex and a subgraph of $G$, which allows us to use $v$ in concepts which are defined generally for subgraphs of $G$.

If $S$ is a set with an element $x$, we write $S - x = S \setminus \{x\}$.

\section{A generalized reducible subgraph framework}
\label{sec:framework}
Dvo\v{r}\'ak, Masa\v{r}\'ik, Mus\'ilek, and Pangr\'ac \cite{DMMP6}
developed a reducible subgraph framework
which can prove that certain graph classes are flexibly choosable. We lay out some key concepts of their framework with slightly modified terminology.
Given a graph $G$ with an induced subgraph $H$, for each integer $k \geq 3$, we define the function $\ell_{H,k}:V(H) \rightarrow \mathbb Z$ so that \[\ell_{H,k}(v) = k - \deg_G(v) + \deg_H(v)\]
for each $v \in V(H)$. Note that if $L$ is a $k$-assignment on $G$
and an $L$-coloring of $G \setminus H$ is fixed, then for each vertex $v \in V(H)$, $\ell_{H,k}(v)$ gives a lower bound for the number of available colors in $L(v)$.

\begin{definition}
    Let $H$ be a graph, let $k \geq 3$ an integer, 
and let $f:V(H) \rightarrow \mathbb Z$. 
We say that $H$ is \emph{weakly $(f,k)$-reductive} if 
 for every $f$-assignment $L$ on $H$, the following properties (FIX') and (FORB-$t$) hold, with $t = k-2$:
    \begin{enumerate}
        \item[(FIX')] \label{item:fix} For each vertex $v \in V(H)$ and color $c \in L(v)$, there exists an $L$-coloring $\phi$ of $H$ for which $\phi(v) = c$, and 
        \item[(FORB-$t$)] \label{item:2-forb} For each $t$-tuple $v_1,\dots,v_t \in V(H)$ and color $c \in L(v_1) \cup \dots \cup L(v_t)$, 
        there exists an $L$-coloring $\phi$ of $H$ such that $\phi(v_j) \neq c$ for all $1 \leq j \leq t$.
    \end{enumerate}
    If $H$ is an induced subgraph of another graph $G$, then we say that $H$ is a \emph{weakly $k$-reducible subgraph} of $G$ if $H$ is weakly $(\ell_{H,k},k)$-reductive.
\end{definition}

Note that the vertices $v_1, \dots, v_t$ in (FORB-$t$) are not necessarily distinct, so (FORB-$t$) is nontrivial on graphs with fewer than $t$ vertices, and for $t \geq 1$, (FORB-$t$) implies (FORB-$(t-1)$).
We note that the (FIX') implies (FORB-$1$) whenever $f(v) \geq 2$ for each vertex $v \in V(H)$.

Dvo\v{r}\'ak, Masa\v{r}\'ik, Mus\'ilek, and Pangr\'ac showed that if every induced subgraph of a graph $G$ has a weakly $k$-reducible subgraph on at most $b$ vertices, then $G$ is $\epsilon$-flexibly $k$-choosable for some constant $\epsilon > 0$ depending on $k$ and $b$.
They used this fact to prove the existence of a constant $\epsilon > 0$ such that every planar graph of girth $6$ is $\epsilon$-flexibly $3$-choosable \cite{DMMP6}.
In order to prove Theorem \ref{thm:mad113}, we use a generalized version of their framework which allows reducible subgraphs to be arbitrarily large. The framework that we use here is a special case of a more general framework introduced by the authors in \cite{RBPB}.

\begin{definition}
\label{def:fixforb}
Let $H$ be a graph, let $k \geq 3$ be an integer, 
and let $f:V(H) \rightarrow \mathbb Z$. Given $\alpha > 0$, we say that $H$ is \emph{$(f,k,\alpha)$-reductive} if 
 for every $f$-assignment $L$ on $H$, there exists a probability distribution on $L$-colorings $\phi$ of $H$ such that the following hold:
\begin{enumerate}
    \item[(FIX)] for each $v \in V(H)$ and each color $c \in L(v)$, $\Pr(\phi(v) = c) \geq \alpha$;
    \item[(FORB)] for each subset $U \subseteq V(H)$ of size at most $k - 2$ and each color $c \in \bigcup_{u \in U} L(u)$, $\Pr(\phi(u) \neq c \ \forall u \in U) \geq \alpha$.
\end{enumerate}
If $H$ is an induced subgraph of another graph $G$, then we say that $H$ is a \emph{$(k,\alpha)$-reducible subgraph} of $G$ if $H$ is $(\ell_{H,k},k,\alpha)$-reductive.
\end{definition}

Observe that the $(k,\alpha)$-reducibility of an induced subgraph $H$ of $G$ is uniquely determined by the graph $H$ and the function $\deg_G$ restricted to $V(H)$.
Note also that the existence of a probability distribution on proper $L$-colorings $\phi$ of $H$ implies that there exists a set $\Phi$ of proper $L$-colorings $\phi$ of $H$ with probability measure $1$. In particular, $\Phi$ is nonempty, so $H$ is $L$-colorable. Therefore, if $H$ is $(f,k,\alpha)$-reductive, then
$H$ is $f$-choosable.
Furthermore, if a triple $(H,k,f)$ satisfies (FORB), then $f(v) \geq 2$ for each vertex $v \in V(H)$.

The following lemma shows that given a graph $G$ and a function $f:V(G) \rightarrow \mathbb Z$ bounded above by $k$, an induced subgraph of $G$ on at most $b$ vertices which is weakly $k$-reducible is
also $(k,\alpha)$-reducible for some constant $\alpha > 0$ depending on $k$ and $b$.

\begin{lemma}
\label{lem:weak-to-strong}
    Let $H$ be a graph on at most $b$ vertices. Let $k \geq 3$ be an integer, and let $f:V(H) \rightarrow \mathbb Z$ satisfy $f(v) \leq k$ for each $v \in V(H)$. If $H$ is weakly $(f,k)$-reductive, then $H$ is $(f,k,\alpha)$-reductive for the constant $\alpha = k^{-b}$.
\end{lemma}
\begin{proof}
    Let $L$ be an $f$-assignment on $H$, and let $\phi$ be an $L$-coloring of $H$ chosen uniformly at random from all $L$-colorings of $H$. As $|L(v)| \leq k$ for each vertex $v \in V(H)$, and as $|V(H)| \leq b$, each $L$-coloring $\phi$ of $H$ is chosen with probability at least $\alpha = k^{-b}$. Now, consider a vertex $v \in V(H)$ and a color $c \in  L(v)$. By (FIX'), some $L$-coloring $\phi$ of $H$ assigns $c$ to $v$. Therefore, $c$ is assigned to $v$ with probability at least $\alpha$. Similarly, let $U \subseteq V(H)$ be a vertex subset of size at most $k-2$, and let $c \in \bigcup_{u \in U} L(u)$. Letting $t = |U|$, the (FORB-$t$) condition implies that some $L$-coloring $\phi$ of $H$ avoids the color $c$ at all vertices of $U$. Therefore, $c$ is avoided at all vertices of $U$ with probability at least $\alpha$. Hence, $H$ is $(f,k,\alpha)$-reductive.
\end{proof}

In \cite{DMMP6}, Masa\v{r}\'ik, Mus\'ilek, and Pangr\'ac say that an induced subgraph $H$ of $G$ is $k$-reducible 
if
the properties in (FIX') and (FORB-$(k-2)$) hold for the function $\ell_{H,k}$.
However, in practice, they only consider reducible subgraphs with at most $b$ vertices, for some constant $b$.
Therefore, if a graph $H$ with at most $b$ vertices is $k$-reducible by their definition, then
given an $\ell_{H,k}$-assignment $L$ on $H$,
a uniform distribution on all $L$-colorings of $H$ guarantees that (FIX) and (FORB) both hold with the value $\alpha = k^{-b}$. Thus, $H$ is also $(\ell_{H,k},k,k^{-b})$-reductive and hence $(k,k^{-b})$-reducible in our 
framework.

Before showing the main application of our definitions, we need the following probabilistic lemma.

\begin{lemma}
\label{lem:conditional}
    Let $A_1, \dots, A_t$ be disjoint events in a probability space with nonzero probability. Then, for each event $X$,
    \[\Pr(X|A_1 \cup \dots \cup A_t) = \sum_{i=1}^t \Pr(X|A_i) P(A_i | A_1 \cup \dots \cup A_t).\]
\end{lemma}
\begin{proof}
    By the definition of conditional probability,
    \begin{eqnarray*}
    \sum_{i=1}^t \Pr(X|A_i) P(A_i | A_1 \cup \dots \cup A_t)& =& \sum_{i=1}^t \frac{\Pr (X \cap A_i)}{\Pr(A_i)} \cdot \frac{\Pr(A_i)}{\Pr(A_1 \cup \dots \cup A_t)}  \\
    &=& \frac{\Pr(X \cap (A_1 \cup \dots \cup A_t) )}{\Pr(A_1 \cup \dots \cup A_t)} \\
    &=& \Pr(X | A_1 \cup \dots \cup A_t).
    \end{eqnarray*}
\end{proof}

The following lemma is our main tool for proving Theorem
\ref{thm:mad113}. 
This lemma is an analogue of \cite[Lemma 4]{DMMP-triangle-free}, and these two lemmas have nearly identical proofs.

\begin{lemma}
\label{lem:main-k-red}
    For each integer $k \geq 3$ and real number $\alpha > 0$, there exists a value $\epsilon = \epsilon(k,\alpha) > 0$ as follows. Let $G$ be a graph. If for every $Z \subseteq V(G)$, the graph $G[Z]$ contains an induced $(k,\alpha)$-reducible subgraph, then $G$ is
    weighted $\epsilon$-flexibly $k$-choosable.
\end{lemma}
\begin{proof}
Let $G$ be a graph, and let $L$ be a $k$-assignment on $G$.
We define $\epsilon = \frac{2\alpha^{k-1}}{k}$.
We aim to prove the stronger statement that there exists a distribution on $L$-colorings $\phi:V(G) \rightarrow \mathbb N$ such that the following two properties hold:
\begin{itemize}
    \item For each $v \in V(G)$ and color $c \in L(v)$, $\Pr(\phi(v) = c) \geq \epsilon$;
    \item For each subset $U \subseteq V(G)$  of size at most $k-2$ and each color $c \in \bigcup_{u \in U}L(u)$, $\Pr(\phi(u) \neq c \ \forall u \in U) \geq \alpha^{|U|}$.
\end{itemize}
If a distribution satisfying these two properties exists, Lemma \ref{lem:prob} implies that $G$ is weighted $\epsilon$-flexibly $k$-choosable.

We prove the existence of our distribution on
$L$-colorings of $G$ by 
induction on $|V(G)|$. 
If $|V(G)|  = 0$, then the distribution that assigns the empty coloring with probability $1$ satisfies both properties. Hence, we assume that $|V(G)| \geq 1$. By the lemma's assumption, $G$ contains an induced $(k,\alpha)$-reducible subgraph $H$. By the induction hypothesis, there exists a distribution on $L$-colorings $\phi:V(G \setminus H) \rightarrow \mathbb N$ that satisfies our two properties.

Now, we describe a random procedure for obtaining an $L$-coloring of $G$.
First, we choose an $L$-coloring $\phi$ of $G \setminus H$ according to the distribution above. Then, we extend $\phi$ to an $L$-coloring of $G$ as follows.
For each $v \in V(H)$, let $L'(v)$ be the set of colors $c \in L(v)$ such that no vertex $w \in N_{G \setminus H}(v)$ satisfies $\phi(w) = c$. We observe that 
\[|L'(v)| \geq k - |N_{G \setminus H}(v)| = k - \deg_G(v) + \deg_H(v) = \ell_{H,k}(v).\]
We delete a set of $|L'(v)| - \ell_{H,k}(v)$ colors from $L'(v)$ chosen uniformly at random, so that $|L'(v)| = \ell_{H,k}(v)$. Then, $L'$ is an $\ell_{H,k}$-assignment on $H$.
As $H$ is $(k,\alpha)$-reducible, there exists a distribution on $L'$-colorings $\psi$ of $H$ such that the properties (FIX) and (FORB) hold with the constant $\alpha$.
We choose an $L'$-coloring $\psi$ of $H$ from this distribution, and we 
take the union $\phi \cup \psi$.
By the definition of the list assignment $L'$, the coloring $\phi \cup \psi$  is a proper $L$-coloring of $G$.

Now, we show that the distribution on $L$-colorings $\phi \cup \psi$ of $G$ that we obtain  from our random procedure satisfies our two properties.
\begin{enumerate}
    \item 
For the first property, let $v \in V(G)$ and $c \in L(v)$. If $v \in V(G \setminus H)$, then by the induction hypothesis, $\Pr(\phi(v) = c) \geq \alpha > \epsilon$. If $v \in V(H)$, then let $U = N_{G \setminus H}(v)$. As $H$ is $(k,\alpha)$-reducible, it follows that $\ell_{H,k}(v) = k - |N_{G\setminus H}(v)| \geq 2$, and hence $|U| \leq k-2$.
Thus, by the induction hypothesis, with probability at least $\alpha^{k-2}$, $\phi(u) \neq c$ for each $u \in U$.

Next, suppose it is given that $\phi = \phi_0$, where $\phi_0$ is a fixed $L$-coloring of $G \setminus H$ such that $\phi_0(u) \neq c$ for all $u \in U$. With $\phi = \phi_0$ given, the conditional probability that $c \in L'(v)$ is at least $\frac{\ell_{H,k}(v)}{k} \geq \frac{2}{k}$. Then, as $H$ is $(k,\alpha)$-reducible, 
the subsequent conditional probability that $\psi(v) = c$ is at least $\alpha$. Therefore, with $\phi = \phi_0$ given, the conditional probability that $\psi(v) = c$ is at least $\frac{2\alpha}{k}$.

Now, let $\Phi$ be the set of all fixed $L$-colorings $\phi_0$ of $G \setminus H$ for which $\phi_0(u) \neq c$ for all $u \in U$. By Lemma \ref{lem:conditional}, 
\begin{eqnarray*}
\Pr (\psi(v) = c | \phi(u) \neq c \  \forall u \in U)  &=&
\Pr \left (\psi(v) = c \biggr \rvert \bigcup_{\phi_0 \in \Phi} (\phi = \phi_0) \right ) \\
&=& \sum_{\phi_0 \in \Phi} \Pr(\psi(v) = c| \phi = \phi_0) \Pr(\phi = \phi_0 | \phi \in \Phi) \\
&\geq & \frac{2\alpha}{k} \sum_{\phi_0 \in \Phi} \Pr(\phi = \phi_0 | \phi \in \Phi) \\
&=& \frac{2\alpha}{k}.
\end{eqnarray*}
As $\phi(u) \neq c$ for all $u \in U$ with probability at least $\alpha^{k-2}$,
our random coloring $\phi \cup \psi$
assigns
$c$ to $v$ with probability at least 
$\alpha^{k-2} \left ( \frac{2 \alpha}{k} \right ) =\epsilon$.
\item For the second property, let $U \subseteq V(G)$ be a subset of size at most $k-2$, and let $c \in \bigcup_{u \in U} L(u)$.
Write $U_1 = U \cap V(G \setminus H)$ and $U_2 = U \cap V(H)$. By the induction hypothesis, our random coloring $\phi$ of $G \setminus H$ avoids $c$ at each $u \in U_1$ with probability at least $\alpha^{|U_1|}$. Hence, if $U = U_1$, then we are done. Otherwise, $|U_2| \geq 1$, and the argument proceeds as follows.

Suppose that it is given that $\phi = \phi_0$, where $\phi_0$ is a fixed $L$-coloring of $G \setminus H$
such that $\phi_0(u) \neq c$ for all $u \in U_1$. Then, as $H$ is $(k,\alpha)$-reducible, the conditional probability that $\psi(u) \neq c$ for all $u \in U_2$ is at least $\alpha$. Therefore, with $\phi=\phi_0$ given, the conditional probability that $\psi(u) \neq c$ for all $u \in U_2$ is at least $\alpha$.

Now, let $\Phi$ 
be the set of all fixed $L$-colorings $\phi_0$ of $G \setminus H$ for which $\phi_0(u) \neq c$ for all $u \in U_1$. By Lemma \ref{lem:conditional},
\begin{eqnarray*}
   & &  \Pr\left (\psi(u) \neq c \ \forall u \in U_2 \biggr |  \phi(u) \neq c \ \forall u \in U_1 \right ) \\
    &=& \Pr \left (\psi(u) \neq c \ \forall u \in U_2 | \bigcup_{\phi_0 \in \Phi} (\phi = \phi_0) \right ) \\
    &=& \sum_{\phi_0 \in \Phi}  \Pr \left (\psi(u) \neq c \ \forall u \in U_2 | (\phi = \phi_0) \right )  \Pr(\phi = \phi_0 | \phi \in \Phi) \\
    &\geq & \alpha \sum_{\phi_0 \in \Phi}    \Pr(\phi = \phi_0 | \phi \in \Phi) \\
    &=& \alpha.
\end{eqnarray*}
As $\phi(u) \neq c$ for all $u \in U_1$ with probability at least $\alpha^{|U_1|}$, our final coloring $\phi \cup \psi$ avoids $c$ at all $u \in U$ with probability at least $\alpha \cdot \alpha^{|U_1|} \geq \alpha^{|U_2| + |U_1|} = \alpha^{|U|}$.
\end{enumerate}
Hence, we have a distribution on $L$-colorings
that satisfies both properties, and the proof is complete.
\end{proof}

\section{Tools for identifying reductive subgraphs}
\label{sec:tools}

\subsection{Some basic lemmas}
In this subsection, we introduce some definitions and lemmas that will
help us identify reductive graphs.
If $G$ is a connected graph, a vertex $v \in V(G)$ is called a \emph{cut vertex} if $G -v$ has at least two components.
A \emph{block} of $G$ is a
maximal connected subgraph $B$ for which the graph $G[B]$ has no cut vertex. In other words,
an
induced subgraph $B \subseteq G$ is a block if $B$ is maximal with respect to the property of being $2$-connected or being isomorphic to $K_1$ or $K_2$. 
A \emph{terminal block} of $G$ is a block containing at most one cut vertex of $G$.

A \emph{Gallai tree} (introduced by Gallai \cite{Gallai}) is a connected graph for which every block is a clique or odd cycle. 
Erd\H{o}s, Rubin, and Taylor \cite[Theorem R]{ERT} 
proved that a graph $T$ is not a Gallai tree if and only if $T$ contains either an even cycle or a theta graph (i.e.~a cycle with a single chord) as an induced subgraph.
The first lemma that we present
was also proven
by Erd\H{o}s, Rubin, and Taylor \cite{ERT}.
\begin{lemma}
\label{lem:ERT}
If $H$ is a connected graph and
$f:V(H) \rightarrow \mathbb Z$ satisfies $f(v) \geq \deg(v)$ for each $v \in V(H)$, then $H$ is $f$-choosable if and only if one of the following conditions holds:
\begin{itemize}
    \item $f(u) > \deg(u)$ for at least one $u \in V(H)$;
    \item $H$ is not a Gallai tree.
\end{itemize}
\end{lemma}

Given a graph $H$, a vertex subset $U \subseteq V(H)$, and a function $f:V(H) \rightarrow \mathbb \mathbb Z$, we define the function $(f - 1_U): V(H) \rightarrow \mathbb Z$ so that 
\[
(f - 1_U)(v) = 
\begin{cases} v -1 &  \textrm { if $v \in U$} \\
v & \textrm{ if $v \not \in U$} 
\end{cases}
.\]
Equivalently, $1_U(v)$ is the function that equals $1$ when $v \in U$ and equals $0$ elsewhere.
(This notation extends notation of Dvo\v{r}\'ak, Masa\v{r}\'ik, Mus\'ilek, and Pangr\'ac \cite{DMMP-triangle-free}.)
If $T$ is a Gallai tree, then we say that a function $f:V(T) \rightarrow \mathbb N$ is \emph{bad} on $T$ if there exists a set $U \subseteq V(T)$ of size at most $2$ such that for each $w \in V(T)$, $(f -1_U)(w) = \deg_T(w)$.
We note that if $f$ is bad on $T$, then the property (FORB-$2$) does not hold for $T$ and $f$.

\begin{lemma}
\label{lem:Gallai}
Let $H$ be a graph, and let $R \subseteq V(H)$ be a dominating set of $H$. Let $f:V(H) \rightarrow \mathbb N$ satisfy $f(v) \geq \deg_H(v)$ for $v \in V(H) \setminus R$ and $f(r) \geq \deg_H(r) + 1$ for $r \in R$. Suppose that $H$ has no induced Gallai tree subgraph $T$ for which $f$ is bad on $T$.
Then, for each pair $u,v \in V(H)$, either
$H$ is $f -1_{\{u,v\}}$-choosable, or 
$uv$ is an edge of $H$ with exactly one endpoint in $R$.\end{lemma}
\begin{proof}
    We write $g = f -1_{\{u,v\}}$. We execute the following procedure on $H$:
    \begin{quote}
        While $H$ contains a vertex $v$ for which $g(v) > \deg_H(v)$, delete $v$ from $H$.
    \end{quote}
    We call the output of this procedure $H'$. A greedy inductive argument shows that $H$ is $g$-choosable if and only if $H'$ is $g$-choosable.
    We observe the following:
    \begin{itemize}
        \item No vertex $r \in R \setminus \{u,v\}$ belongs to $V(H)$. 
        \item If a vertex $w \in V(H) \setminus \{u,v\}$ has no neighbor in $R \cap \{u,v\}$, then $w$ does not belong to $H'$.
    \end{itemize}
    Now, suppose the lemma is false. We make the following claim.
    \begin{claim}
    For each vertex $w \in V(H')$, $g(w) \geq \deg_{H'}(w)$. 
    \end{claim}
    \begin{proof}
        For each $w \in V(H') \setminus \{u,v\}$, $g(w) = f(w) \geq \deg_H(w) \geq \deg_{H'}(w)$. Now, we consider three cases for $u$ and $v$.

        If $u,v \in R$, then for each $z \in \{u,v\} \cap V(H')$, $g(z) = f(z) - 1 \geq \deg_H(z) + 1 - 1 \geq \deg_{H'}(z)$. If $u,v \in V(H) \setminus R$, then $u$ and $v$ each have a neighbor in $R$ which does not appear in $H'$. Therefore, for each $z \in \{u,v\}  \cap V(H') $, $g(z) = f(z) - 1 \geq \deg_H(z) - 1 \geq \deg_{H'}(z)$. Finally, if $u \in V(H) \setminus R$ and $v \in R$, then 
        if $uv \in E(G)$, the lemma is proven; otherwise, $uv \not \in E(G)$. Then, if $u \in V(H')$, then $u$ has a neighbor in $R \setminus \{u,v\}$ which does not appear in $H'$, so that $g(u) = f(u) - 1 \geq \deg_H(u) - 1 \geq \deg_{H'}(u)$. Furthermore, if $v \in V(H')$, then $g(v) = f(v) - 1 \geq \deg_H(v) + 1 -1 \geq \deg_{H'}(v)$. This completes the proof.
    \end{proof}

    Now, if $H'$ is not $g$-choosable, then Lemma \ref{lem:ERT} tells us that $H'$ has a component  $T$ isomorphic to a Gallai tree and that $g(w) = \deg_{H'}(w)$ for each $w \in V(H')$. Then, letting $U = \{u,v\} \cap V(H')$, $f$ is bad on $T$, a contradiction.
\end{proof}

We frequently use the following lemma with the function $f = \ell_{H,4}$ to identify small $(k,\alpha)$-reducible subgraphs while proving Theorem \ref{thm:mad113}.

\begin{lemma}
\label{lem:fix}
    Let $H$ be a connected graph, and let $R \subseteq V(H)$ be a dominating set of $H$. Let $f:V(H) \rightarrow \mathbb N$ satisfy $f(r) \geq \deg_H(r) + 1$ for each $r \in R$ and $f(w) \geq \deg_H(w)$ for each $w \in V(H) \setminus R$. Suppose that the following hold:
    \begin{enumerate}
        \item For each connected subset $U \subseteq V(H)$ of size at most $2$, either $|U| = 2$ and $U \subseteq R$, or each component of $H \setminus U$ contains a vertex $r \in R$.
        \item $H$ contains no induced Gallai tree subgraph $T$ on which $f$ is bad.
    \end{enumerate}
    Then, $H$ is weakly $(f,4)$-reductive.
\end{lemma}

\begin{proof}
    Let $L$ be an $f$-assignment on $H$.
    If $f(v) \leq 1$ for some vertex $v \in V(H)$, then $f$ is bad on the Gallai tree $v$; 
    therefore, $f(v) \geq 2$ for each $v \in V(H)$.

     First, we show that (FIX') holds for $H$ and $L$.
    Consider a vertex $v \in V(H)$ and a color $c \in L(v)$. If we first color $v$ with $c$, then writing $U = \{v\}$ and $H' = H \setminus U$, each remaining vertex $u \in V(H')$ has at least $\deg_{H'}(u)$ available colors. 
    By (1), at least one vertex $r$ in each component of $H'$ has at least $\deg_{H'}(r) + 1$ available colors, so
    by Lemma \ref{lem:ERT}, we can complete an $L$-coloring $\phi$ of $H$ satisfying $\phi(v) = c$. 

    Next, we show that (FORB-$2$) holds for $H$ and $L$. Consider a vertex pair $u,v \in V(H)$. 
    If $u=v$, then (FIX') implies that $H$ has an $L$-coloring that avoids $c$ at $u=v$. Otherwise, suppose that $u \neq v$.
    If $uv \not \in E(H)$ or $u,v \in R$, then as $H$ has no induced Gallai tree subgraph $T$ on which $f$ is bad, Lemma \ref{lem:Gallai} implies that
    $H$ is $f - 1_{\{u,v\}}$-choosable. Hence, we can find an $L$-coloring $\phi$ of $H$ for which $c \not \in \{\phi(u),\phi(v)\}$. On the other hand, if $uv \in E(G)$ and $\{u,v\} \not \subseteq R$,
    then by (2), $H[\{u,v\}]$ is $(f-1)$-choosable; hence, we assign colors $\phi(u)$ and $\phi(v)$ to $u$ and $v$, respectively, which do not equal $c$.
    Then, we let $U = \{u,v\}$ and $H' = H \setminus U$, and each remaining vertex $w \in V(H')$
    has at least $\deg_{H'}(w)$ available colors. Furthermore,
    by (1), at least one vertex $r$ in each component of $H'$ has at least $\deg_{H'}(r) + 1$ available colors. 
    Therefore, by Lemma \ref{lem:ERT}, we can 
    complete an $L$-coloring $\phi$ of $H$ satisfying $c \not \in \{\phi(u),\phi(v)\}$.
\end{proof}

\subsection{A partition lemma}
The goal of this subsection is to introduce a lemma (Lemma \ref{lem:orig_partition}) that will help us identify reductive subgraphs of a given graph $G$. In order to state our lemma, we first need some definitions.
Given a graph $H$, we say that a \emph{subgraph partition} of $H$ is a family $\mathcal H$ of induced subgraphs of $H$ such that $\{V(H_i): H_i \in \mathcal H\}$ is a partition of $V(H)$.
We say that two parts $H_i,H_j \in \mathcal H$ are \emph{adjacent} if some edge of $H$ has an endpoint in $H_i$ and an endpoint in $H_j$.

\begin{definition}
\label{def:superscript}
    Given a graph $H$, a 
    function $f \colon V(H) \rightarrow \mathbb{Z}$, and a subset $U \subseteq V(H)$, we
    define $f^U \colon V(H) \setminus U \rightarrow \mathbb{Z}$ such that $f^{U}(v) = f(v) - |N_H(v) \cap U|$ for all $v \in V(H) \setminus U$. When $X$ is a subgraph of $H$, we define $f^X := f^{V(X)}$. 
\end{definition}

We note that in the definition above, if $L$ is an $f$-assignment on $H$, then after giving $H[U]$ an $L$-coloring,
each vertex $v \in V(H) \setminus U$ has at least $f^U(v)$ available colors in $L(v)$.

\begin{definition}
\label{def:strong-part}
Let $H$ be a graph, and let $\mathcal{H} = \{H_1, \ldots, H_m\}$ be a subgraph partition of $H$.
Let $f \colon V(H) \rightarrow \mathbb{Z}$ be a 
function on $H$. 
We say that $H_i \in \mathcal H$ is an \emph{$f$-strong part} of $\mathcal H$ if the following hold:
\begin{enumerate}
    \item \label{item:forb2} For each $f$-assignment $L$ on $H_i$, (FIX') and (FORB-2) hold;
    \item \label{item:strong-forb1} For each part $H_j \in \mathcal H  - H_i$ and each $f^{H_j}$-assignment $L$ on $H_i$, (FORB-1) holds;
    \item \label{item:color-last} 
    $H_i$ is $f^{H \setminus H_i}$-choosable.
\end{enumerate}
\end{definition}
\begin{definition}
\label{def:weak-part}
Let $H$ be a graph, and let $\mathcal{H} = \{H_1, \ldots, H_m\}$ be a subgraph partition of $H$.
Let $f \colon V(H) \rightarrow \mathbb{Z}$ be a 
function on $H$.
 We say that a subgraph $H_i \in \mathcal H$ is an \emph{$f$-weak part} of $\mathcal H$ if
 \begin{enumerate}
     \item For each $f$-assignment $L$ on $H_i$,  (FIX') and (FORB-2) hold,
     \item 
           \label{item:color-third}
     For each pair $H_j, H_{j'} \in \mathcal{H} - H_i$,
     $H_i$ is $f^{H_j \cup H_{j'}}$-choosable.
\end{enumerate}
\end{definition}

Note that if $H_i$ is an $f$-strong part of $\mathcal H$, then $H_i$ is also an $f$-weak part of $\mathcal H$.
The following lemma shows that if $H$ has 
a subgraph partition $\mathcal H = \{H_1, \dots, H_m\}$ in which every part $H_i \in \mathcal H$ is weak and at least $m-1$ parts of
$\mathcal H$ are strong,
then $H$ is a reductive graph. This lemma is a critical tool for our proof of Theorem \ref{thm:mad113}.

\begin{lemma}
\label{lem:orig_partition}
    Let $d \geq 4$, and let $H$ be a graph for which
    each vertex $v \in V(H)$ satisfies $\deg(v) \leq d$.
    Let $f:V(H) \rightarrow \mathbb Z$ be a function satisfying $f(v) \leq 4$ for each $v \in V(H)$. Suppose that there exists a subgraph partition $\mathcal{H} = \{H_1, \ldots, H_m\}$ of $H$ and an integer $b \geq 4$
    such that the following hold:
    \begin{itemize}

        \item $H_1$ is a $f$-weak part of order at most $b$;

        \item 
        $H_i$ is a $f$-strong part of order at most $b$ for each $2 \leq i \leq m$.
    \end{itemize}
    Then $H$ is $(f, 4, (bd)^{-4}4^{-2b})$-reductive.
\end{lemma}
\begin{proof}
    We define an auxiliary graph $A$ for which $V(A) = \mathcal H$. We let two parts $H_i,H_j \in \mathcal H$ be adjacent in $A$ if and only if $H_i$ and $H_j$ are adjacent parts of $\mathcal H$.
    As each part $H_i$ consists of at most $b$ vertices of degree at most $d$, the maximum degree of $A$ is at most $bd$. As $b \geq 4$ and $d \geq 4$, in particular $bd \not \in \{2,3,7,57\}$; therefore, by \cite[Equation 1]{AlonMohar}, $\chi(A^2) \leq (bd)^2$.
    Thus, we use the set $\{1,\dots,(bd)^2\}$
    to label each part
    $H_i \in \mathcal{H}$ with a label $\tau(H_i)$ such that
    $\tau(H_i) \neq \tau(H_j)$ for each adjacent pair $H_i, H_j \in \mathcal H$, and so that
    no part $H_i \in \mathcal{H}$ has distinct 
    adjacent parts $H_j$ and $H_{j'}$ satisfying $\tau(H_j) = \tau(H_{j'})$. 
    (A simpler greedy argument shows that $(bd)^2 + 1$ distinct labels suffice; however, we use $(bd)^2$ labels for the sake of algebraic simplicity.)

    Now, given an $f$-assignment $L$ on $H$, we give $H$ a random $L$-coloring using the following procedure. We first define the procedure and then show that it successfully produces an $L$-coloring of $H$.
    \begin{itemize}
    \item First, we choose a value $t_1 \in \{1,\dots,(bd)^2\}$ uniformly at random. For each part $H_i \in \mathcal{H}$ with label $\tau(H_i) = t_1$, we choose an $L$-coloring of $H_i$ uniformly at random. %
    \item Next, we choose a value $t_2 \in \{1,\dots,(bd)^2\} -t_1$ uniformly at random.
    For each part $H_i \in \mathcal{H}$ with label $\tau(H_i) = t_2$, we 
    give $H_i$ an $L$-coloring chosen uniformly at random from the set of $L$-colorings of $H_i$ that do not create a monochromatic edge in $H$.
    \item Next, if $\tau(H_1) \not \in \{t_1, t_2\}$, we let $t_3 = \tau(H_1)$; otherwise, we choose $t_3$ uniformly at random from $\{1, \dots, (bd)^2\} \setminus \{t_1,t_2\}$. Then, for each part $H_i \in \mathcal{H}$ with label $\tau(H_i) = t_3$, we give $H_i$ an $L$-coloring chosen uniformly at random from the set of $L$-colorings of $H_i$ that do no create a monochromatic edge in $H$. 
    \item Finally, 
    we give an arbitrary order to the remaining uncolored parts $H_i \in \mathcal H$. Then,
    we consider each uncolored part $H_i \in \mathcal{H}$ in order, and we give $H_i$ an $L$-coloring chosen uniformly at random from the set of $L$-colorings of $H_i$ that do not create a monochromatic edge in $H$. 
    \end{itemize}

    We claim that this procedure produces a proper $L$-coloring of $H$.
    To this end, we check that each part $H_i \in \mathcal H$ is successfully colored.
    When $H_1$ is colored, no more than two parts adjacent to $H_1$ have been colored. Hence, 
    as $H_1$ is an $f$-weak part of $\mathcal H$, condition (\ref{item:color-third}) of Definition \ref{def:weak-part} 
    implies that
    $H_1$ is colored successfully. Additionally, as each part $H_i \in \mathcal{H} - H_1$ is $f$-strong,
    condition (\ref{item:color-last}) of 
    Definition \ref{def:strong-part}  implies that
    $H_i$ is $f^{H \setminus H_i}$-choosable, and hence our procedure successfully gives an $L$-coloring to $H_i$. So, for each $H_i \in \mathcal{H} - H_1$, the procedure successfully colors $H_i$, and hence, the procedure above produces an $L$-coloring of $H$. We call this $L$-coloring $\phi$. 

    Next, we demonstrate that the 
    random coloring $\phi$ of $H$ produced by our procedure has a distribution satisfying
    (FIX) and (FORB) with the constant $\alpha = (bd)^{-4} 4^{-2b}$.
    We repeatedly use the fact that each part $H_i \in \mathcal H$ has at most $b$ vertices and hence at most $4^b$ $L$-colorings.
    For (FIX), let $v \in V(H)$ and $c \in L(v)$.
    Write $H_i$ for the part of $\mathcal H$ containing $v$.
    With probability at least $(bd)^{-2}$, $t_1 = \tau(H_i)$. Then, because (FIX') holds for $H_i$, $v$ is colored with $c$ with conditional probability at least $4^{-b}$. Hence, $\phi(v) = c$ with probability at least $(bd)^{-2} 4^{-b} > \alpha$.

    Now, we show the distribution on $\phi$ satisfies (FORB) with $\alpha = (bd)^{-4} 4^{-2b}$. Fix $u, v \in V(H)$ and $c \in L(u) \cup L(v)$. Let $u \in V(H_i)$ and $v \in V(H_j)$. We consider two cases.
    \begin{enumerate}
        \item If $\tau(H_i) = \tau(H_j)$, then with probability at least $(bd)^{-2}$, $t_1 = \tau(H_i) = \tau(H_j)$. 
        If $H_i \neq H_j$, then $H_i$ and $H_j$ are not adjacent.
        By applying (FORB-$2$) to $H_i$ and $H_j$ using condition (\ref{item:forb2}) of Definitions \ref{def:strong-part} and \ref{def:weak-part},
         $\phi$ avoids $c$ at both $u$ and $v$
        with conditional probability at least $4^{-2b}$. 
        On the other hand, if $H_i = H_j$, then
        $\phi$ avoids $c$ at both $u$ and $v$ with conditional probability at least $4^{-b}$. %
        Hence, with probability at least $(bd)^{-2} 4^{-2b} > \alpha$, $c \not \in \{\phi(u),\phi(v)\}$.

        \item On the other hand, suppose $\tau(H_i) \neq \tau(H_j)$. We assume without loss of generality that $H_1 \neq H_j$, so that $H_j$ is an $f$-strong part of $\mathcal H$. With probability greater than $(bd)^{-4}$, $t_1 = \tau(H_i)$ and $t_2 = \tau(H_j)$. 
        Condition (\ref{item:forb2}) of Definitions \ref{def:strong-part} and \ref{def:weak-part} implies that (FORB-$2$) holds for 
        $H_i$,  so
        $\phi(u) \neq c$ with conditional probability at least $4^{-b}$. 
        Then, 
        condition (\ref{item:strong-forb1}) of Definition \ref{def:strong-part} implies that
        (FORB-$1$)  holds for  $H_j$ with the available colors, 
        so $\phi(v) \neq c$ 
        with conditional probability at least $4^{-b}$. Overall, $c \not \in \{\phi(u),\phi(v)\}$
        with probability at least $(bd)^{-4} 4^{-2b}$. 
    \end{enumerate}

    Overall, our procedure satisfies (FIX) and (FORB) with a probability constant $\alpha = (bd)^{-4} 4^{-2b}$, and hence $H$ is $(f, 4,(bd)^{-4}4^{-2b})$-reductive.
\end{proof}

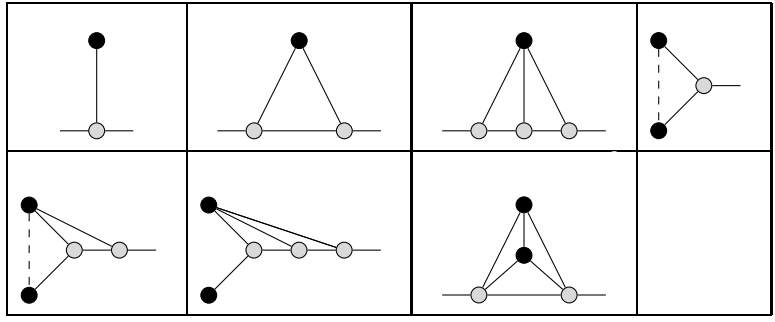
\begin{figure}
\begin{tabular}{ |c|c|c|c| } 
 \hline
\begin{tikzpicture}
[scale=1.2,auto=left,every node/.style={circle,fill=gray!30,minimum size = 6pt,inner sep=0pt}]
\node(r1) at (0,1) [draw = black, fill = black] {};
\node(p1) at (0,0) [draw = black] {};
\node(h1) at (-0.5,0) [draw=white,fill=white]  {};
\node(h2) at (0.5,0) [draw=white,fill=white] {};
\foreach \from/\to in {r1/p1,p1/h1,p1/h2}
    \draw (\from) -- (\to);
\end{tikzpicture} & \begin{tikzpicture}
[scale=1.2,auto=left,every node/.style={circle,fill=gray!30,minimum size = 6pt,inner sep=0pt}]
\node(r1) at (0,1) [draw = black, fill = black] {};
\node(z) at (0.5,1.4) [draw=white,fill=white,minimum size=0pt,inner sep=0pt] {};
\node(p1) at (-0.5,0) [draw = black] {};
\node(p3) at (0.5,0) [draw = black] {};
\node(h1) at (-1,0) [draw=white,fill=white]  {};
\node(h2) at (1,0) [draw=white,fill=white]  {};
\foreach \from/\to in {r1/p1,r1/p3,p1/p3,p1/h1,p3/h2}
    \draw (\from) -- (\to);
\end{tikzpicture} & \begin{tikzpicture}
[scale=1.2,auto=left,every node/.style={circle,fill=gray!30,minimum size = 6pt,inner sep=0pt}]
\node(r1) at (0,1) [draw = black, fill = black] {};
\node(p1) at (-0.5,0) [draw = black] {};
\node(p2) at (0,0) [draw = black] {};
\node(p3) at (0.5,0) [draw = black] {};
\node(h1) at (-1,0) [draw=white,fill=white] {};
\node(h2) at (1,0) [draw = white,fill=white] {};
\foreach \from/\to in {r1/p1,r1/p2,r1/p3,p1/p2,p2/p3,p3/h2,p1/h1}
    \draw (\from) -- (\to);
\end{tikzpicture}  & \begin{tikzpicture}
[scale=1.2,auto=left,every node/.style={circle,fill=gray!30,minimum size = 6pt,inner sep=0pt}]
\node(r1) at (0,0.5) [draw = black, fill = black] {};
\node(r2) at (0,-0.5) [draw = black, fill = black] {};
\node(p3) at (0.5,0) [draw = black] {};
\node(h2) at (1,0) [draw = white,fill=white] {};
\draw [dashed] (r1) -- (r2);
\foreach \from/\to in {r1/p3,p3/h2,r2/p3}
    \draw (\from) -- (\to);
\end{tikzpicture}
\\ \hline
\begin{tikzpicture}
[scale=1.2,auto=left,every node/.style={circle,fill=gray!30,minimum size = 6pt,inner sep=0pt}]
\node(r1) at (0,0.5) [draw = black, fill = black] {};
\node(r2) at (0,-0.5) [draw = black, fill = black] {};
\node(p1) at (0.5,0) [draw = black] {};
\node(p2) at (1,0) [draw = black] {};
\node(h2) at (1.5,0) [draw=white,fill=white]  {};
\draw [dashed] (r1) -- (r2);
\foreach \from/\to in {r1/p3,p1/p2,r2/p1,p2/h2,r1/p2}
    \draw (\from) -- (\to);
\end{tikzpicture} & 

    \begin{tikzpicture}
[scale=1.2,auto=left,every node/.style={circle,fill=gray!30,minimum size = 6pt,inner sep=0pt}]
\node(r1) at (0,0.5) [draw = black, fill = black] {};
\node(r2) at (0,-0.5) [draw = black, fill = black] {};
\node(p1) at (0.5,0) [draw = black] {};
\node(p2) at (1,0) [draw = black] {};
\node(p3) at (1.5,0) [draw = black] {};
\node(h2) at (2,0) [draw=white,fill=white] {};
\foreach \from/\to in {r1/p3,r1/p1,p1/p2,p2/p3,r2/p1,p3/h2,r1/p2,r1/p3}
    \draw (\from) -- (\to);
\end{tikzpicture}
& 
\begin{tikzpicture}
[scale=1.2,auto=left,every node/.style={circle,fill=gray!30,minimum size = 6pt,inner sep=0pt}]
\node(r1) at (0,1) [draw = black, fill = black] {};
\node(p1) at (-0.5,0) [draw = black] {};
\node(p2) at (0,0.44) [draw = black,fill=black] {};
\node(p3) at (0.5,0) [draw = black] {};
\node(h1) at (-1,0)[draw=white,fill=white]  {};
\node(h2) at (1,0) [draw=white,fill=white] {};
\node(z) at (1,1.5) [draw=white,fill=white] {};
\foreach \from/\to in {r1/p1,r1/p2,r1/p3,p1/p2,p2/p3,p3/h2,p1/p3,p1/h1}
    \draw (\from) -- (\to);
\end{tikzpicture}
 & 

 \\ 
 \hline
\end{tabular}

\caption{The figure shows the subgraphs of $H$ listed in Lemma \ref{lem:strong-parts}. 
The half-edges drawn in each image indicate that a vertex has a neighbor in $H$ which is not part of the subgraph.
The dashed edges denote an edge that may or may not exist in $H$.
The function $f$ in the lemma satisfies $f(r) \geq \deg_H(r) + 1$
for each vertex $r$ drawn in black and $f(v) \geq \deg_H(v)$ for each vertex $v$ drawn in grey.
}
\label{fig:strong-parts}
\end{figure}

We say that a \emph{petal graph} is a graph obtained from a path $P$ by adding a vertex $r$ adjacent to each vertex of $P$. We say that $r$ is the \emph{root} of the petal graph.
The following lemma describes several strong parts that we often use in the proof of our main result. The strong parts in the following lemma are sketched in Figure \ref{fig:strong-parts}.

\begin{lemma} 
\label{lem:strong-parts}
Let $H$ be a graph with a function $f:V(H) \rightarrow \mathbb Z$, and let $\mathcal H$ be a subgraph partition of $H$.
Then, a part $H_i \in \mathcal H$ is an 
$f$-strong part if 
one of the following holds:  
\begin{enumerate}
    \item \label{item:internal-petal} $H_i$ is a petal graph with a root $r$ and a path $P = (v_1, \dots, v_t)$ satisfying the following properties:
    \begin{itemize}
        \item $1 \leq t \leq 3$;
        \item $f(r) \geq \deg_H(r) + 1$;
        \item $f(v) \geq \deg_H(v)$ for each $v \in V(P)$;
        \item $v_1$ has a neighbor in a part $H_j \in \mathcal H - H_i$;
        \item $v_t$ has a neighbor in a part $H_{j'} \in \mathcal H \setminus \{H_i, H_j\}$.
    \end{itemize}
    \item $H_i$ is obtained from a petal \label{item:strong-end-piece} graph with a root $r$ and a path $P = (v_1, \dots, v_t)$ by adding a neighbor $r' \in N(v_1)$ and possibly the edge $rr'$, so that the following properties are satisfied:
    \begin{itemize}
        \item $1 \leq t \leq 3$;
        \item $f(r) \geq \deg_H(r) + 1$;
        \item $f(r') \geq \deg_H(r') + 1$;
        \item $f(v) \geq \deg_H(v)$ for each $v \in V(P)$;
        \item $v_t$ has a neighbor in $H \setminus H_i$.
    \end{itemize}
    \item \label{item:special-K4} $H_i$ is a $K_4$ with vertices $v_1, v_2, v_3, v_4$ satisfying the following properties:
    \begin{itemize}
        \item $f(v_j) \geq \deg_H(v_j) + 1$ for $j \in \{1,2\}$;
        \item $f(v_j) \geq \deg_H(v_j)$ for $j \in \{3,4\}$;
        \item Each of $v_3, v_4$ has a neighbor in $H \setminus H_i$.
    \end{itemize}
\end{enumerate}
\end{lemma}
\begin{proof}
    We first show that (FORB-$2$) holds for each $f$-assignment on $H_i$. 
    To this end, let $L$ be an $f$-assignment on $H_i$.
    For all graphs $H_i$ outlined in Cases (1)-(3),
     we apply Lemma \ref{lem:fix}. We observe that the set $R$ of vertices $w$ for which $f(w) > \deg_{H_i}(w)$ form a dominating set.
     (In Figure \ref{fig:strong-parts}, the vertices of $R$ are those which are either drawn in black or incident to a half-edge.)

     We claim that each induced Gallai tree subgraph $T$ of $H_i$ 
    either has a  vertex $w \in V(T)$ satisfying $f(v) \geq \deg_T(w) + 2$ or three vertices satisfying $f(w) > \deg_{T}(w)$.
    Indeed, if $T$ is a single vertex $w$, then $f(w) \geq 2  =\deg_T(w) + 2$. If $T$ is a $K_2$, then case analysis shows that some $w \in V(T)$ satisfies $f(w) \geq 3 = \deg_T(w) + 2$. 
    If $T$ is a triangle, then checking each case shows that each $w \in V(T)$ satisfies $f(w) \geq 3 > \deg_T(w)$. 
    If $T$ is a $K_4$, then each vertex $w \in V(T)$ satisfies $f(w) = 4 > \deg_T(w)$.
    If $T$ is obtained by identifying the endpoint of a path $P$ with a vertex of a triangle $K$, then case analysis shows that the endpoint $p \in V(P) \setminus V(K)$ satisifes $f(p) \geq 2 > \deg_T(p)$, and two vertices $w \in V(K)$ satisfy $f(w) > \deg_T(w)$. 
    If $T$ is a path with at least three vertices, $f(w) > \deg_T(w)$ for each $w \in V(T)$.
    Therefore, for each induced Gallai tree subgraph $T$ of $H_i$, 
    $f$ is not bad on $T$.

    We note that each component of $H_i - u$ has a vertex $w$ satisfying $f(w) > \deg_{H_i}(w)$ for each $u \in V(H_i)$, and each component of $H_i \setminus \{u,v\}$
    has a vertex $w$ satisfying $f(w) > \deg_{H_i}(w)$ for each adjacent pair $u,v \in V(H_i)$. 
    Therefore, Lemma \ref{lem:fix} implies that
    $H_i$ is weakly $(f,4)$-reductive, 
    and hence 
    (FORB-$2$) holds.
    This argument also shows that (FIX') holds.

    Next, we show that for each part $H_j \in \mathcal H -H_i$ and $f^{H_j}$-assignment $L$ on $H_i$, (FORB-$1$) holds. To this end, fix an $f^{H_j}$-assignment $L$ on $H_i$. 
    We observe that in each case, for each vertex $v \in V(H_i), |L(v)| \geq 2$.
    Furthermore, for each $u \in V(H_i)$, each component of $H_i - u$ has a vertex $v$ satisfying $f^{H_j}(v) > \deg_{H_i}(v)$. Therefore, given a vertex $u \in V(H_i)$ and color $c \in L(u)$,
    we can give $H_i$ an $L$-coloring $\phi$ for which $\phi(u) \neq c$ by first coloring $u$ with a color from $L(u) -c$ and then using Lemma \ref{lem:ERT} to extend $\phi$ to all of $H_i$. Therefore, (FORB-$1$) holds.

    Finally, as $f(u) \geq \deg_{H}(u)$ for each $u \in V(H_i)$ (and hence $f^{H \setminus H_i}(u) \geq \deg_{H_i}(u)$) and there exists $v \in V(H_i)$ for which $f(v) > \deg_{H}(v)$ (and hence $f^{H \setminus H_i}(v) > \deg_{H_i}(v)$), Lemma \ref{lem:ERT} implies that $H_i$ is $f^{H \setminus H_i}$-choosable. This completes the proof.
\end{proof}

In our proof of the  main result, we repeatedly
use the strong parts in Lemma \ref{lem:strong-parts} along with Lemma \ref{lem:orig_partition} in order to identify reductive graphs.
The next lemmas describe weak parts that often appear in our applications of Lemma \ref{lem:orig_partition}.
We show the weak parts of Lemma \ref{lem:weak-parts} in Figure \ref{fig:weak-parts}.

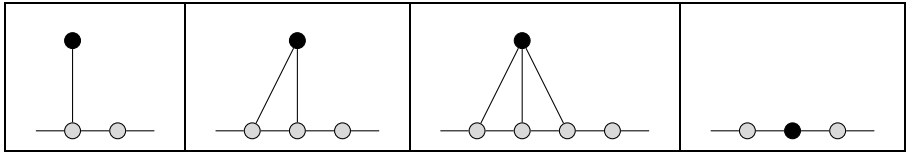
\begin{figure}
\begin{tabular}{ |c|c|c|c| } 
\hline
\begin{tikzpicture}
[scale=1.2,auto=left,every node/.style={circle,fill=gray!30,minimum size = 6pt,inner sep=0pt}]
\node(r1) at (0,1) [draw = black, fill = black] {};
\node(p1) at (0,0) [draw = black] {};
\node(p2) at (0.5,0) [draw = black] {};
\node(h1) at (-0.5,0) [draw=white,fill=white]  {};
\node(h2) at (1,0) [draw=white,fill=white] {};
\foreach \from/\to in {r1/p1,p1/p2,p1/h1,p2/h2}
    \draw (\from) -- (\to);
\end{tikzpicture} & \begin{tikzpicture}
[scale=1.2,auto=left,every node/.style={circle,fill=gray!30,minimum size = 6pt,inner sep=0pt}]
\node(r1) at (0,1) [draw = black, fill = black] {};
\node(z) at (0.5,1.4) [draw=white,fill=white,minimum size=0pt,inner sep=0pt] {};
\node(p1) at (-0.5,0) [draw = black] {};
\node(p3) at (0,0) [draw = black] {};
\node(p4) at (0.5,0) [draw = black] {};
\node(h1) at (-1,0) [draw=white,fill=white]  {};
\node(h2) at (1,0) [draw=white,fill=white]  {};
\foreach \from/\to in {r1/p1,r1/p3,p1/p3,p1/h1,p3/p4,p4/h2}
    \draw (\from) -- (\to);
\end{tikzpicture} & \begin{tikzpicture}
[scale=1.2,auto=left,every node/.style={circle,fill=gray!30,minimum size = 6pt,inner sep=0pt}]
\node(r1) at (0,1) [draw = black, fill = black] {};
\node(p1) at (-0.5,0) [draw = black] {};
\node(p2) at (0,0) [draw = black] {};
\node(p3) at (0.5,0) [draw = black] {};
\node(p4) at (1,0) [draw = black] {};
\node(h1) at (-1,0) [draw=white,fill=white] {};
\node(h2) at (1.5,0) [draw = white,fill=white] {};
\foreach \from/\to in {r1/p1,r1/p2,r1/p3,p1/p2,p2/p3,p3/p4,p4/h2,p1/h1}
    \draw (\from) -- (\to);
\end{tikzpicture} &
\begin{tikzpicture}
[scale=1.2,auto=left,every node/.style={circle,fill=gray!30,minimum size = 6pt,inner sep=0pt}]
\node(z) at (0.5,0) [draw=white,fill=white,minimum size=0pt,inner sep=0pt] {};
\node(p1) at (-0.5,0) [draw = black] {};
\node(p3) at (0,0) [draw = black, fill = black] {};
\node(p4) at (0.5,0) [draw = black] {};
\node(h1) at (-1,0) [draw=white,fill=white]  {};
\node(h2) at (1,0) [draw=white,fill=white]  {};
\foreach \from/\to in {p1/p3,p1/h1,p3/p4,p4/h2}
    \draw (\from) -- (\to);
\end{tikzpicture}
\\
\hline 
\end{tabular}

\caption{The figure shows the subgraphs of $H$ listed in Lemma \ref{lem:weak-parts}. 
The half-edges drawn in each image indicate that a vertex has a neighbor in $H$ that is not part of the subgraph.
The function $f$ in the lemma satisfies $f(r) \geq \deg_H(r) + 1$
for each vertex $r$ drawn in black and $f(v) \geq \deg_H(v)$ for each vertex $v$ drawn in grey.
}
\label{fig:weak-parts}
\end{figure}

\begin{lemma} 
\label{lem:weak-parts}
Let $H$ be a graph with a function $f:V(H) \rightarrow \mathbb Z$, and let $\mathcal H$ be a subgraph partition of $H$.
Then, a part $H_i \in \mathcal H$ is an $f$-weak part if 
one of the following holds: 
\begin{enumerate}
    \item \label{item:weak-petal} $H_i$ is obtained from a petal graph with a root $r$ and a path $P = (v_1, \dots, v_t)$ by deleting the edge $r v_t$, and the following properties are satisfied:
    \begin{itemize}
        \item $2 \leq t \leq 4$;
        \item $f(r) \geq \deg_H(r) + 1$;
        \item $f(v) \geq \deg_H(v)$ for each $v \in V(P)$.
        \item $v_1$ has a neighbor in a part $H_j \in \mathcal H - H_i$;
        \item $v_t$ has a neighbor in a part $H_{j'} \in \mathcal H \setminus \{H_i, H_j\}$.
    \end{itemize}
    \item \label{item:weak-path} $H_i$ is a path $P = (v_1,v_2,v_3)$ satisfying the following properties:
    \begin{itemize}
        \item $f(v_1) \geq \deg_H(v_1)$;
        \item $f(v_2) \geq \deg_H(v_2) + 1$;
        \item $f(v_3) \geq \deg_H(v_3)$;
        \item $v_1$ has a neighbor in a part $H_j \in \mathcal H - H_i$;
        \item $v_3$ has a neighbor in a part $H_{j'} \in \mathcal H \setminus \{H_i, H_j\}$.
    \end{itemize}
\end{enumerate}
\end{lemma}
\begin{proof}
    First, we show that (FIX') holds for every $f$-assignment $L$ on $H_i$. For this, 
    we observe that for each $v \in V(H_i)$, $f(v) \geq \deg_H(v) \geq \deg_{H_i} (v)$. Also,
    for each $u \in V(H_i)$, each component of $H_i - u$ has a vertex $v$ satisfying $f(v) > \deg_{H_i}(v)$. Hence, by Lemma \ref{lem:ERT}, $H_i$ can be $L$-colored even after fixing an arbitrary color $c \in L(u)$ at $u$; therefore, (FIX') holds.

    Next, we show that (FORB-$2$) holds for every $f$-assignment $L$ on $H_i$. We fix two vertices $u, v \in V(H_i)$ and a color $c \in L(u) \cup L(v)$. In each of our lemma's cases,
    if $u$ and $v$ are adjacent, then without loss of generality $f(u) \geq 3$; 
    therefore, even after forbidding $c$ at $u$ and $v$, 
    $u$ and $v$ can be $L$-colored. Then, as each component of $H_i \setminus \{u,v\}$
    has a vertex $w$ satisfying $f(w) > \deg_{H_i}(w)$, Lemma \ref{lem:ERT} implies that we can finish our $L$-coloring of $H_i$.

    Finally, we argue that for each $H_j,H_{j'} \in \mathcal H -H_i$,
    $H_i$ is $f^{H_j \cup H_{j'}}$-choosable. Since $f(v) \geq \deg_{H}(v)$ for all $v \in V(H_i)$ and $f(u) > \deg_H(u)$ for some $u \in V(H_i)$, 
    it follows that $f^{H_j \cup H_{j'}}(v) \geq \deg_{H_i}(v)$ for all $v \in V(H_i)$ and 
    $f^{H_j \cup H_{j'}}(u) > \deg_{H_i}(u)$ for some $u \in V(H_i)$.
    Therefore,
    $H_i$ is $f^{H_j \cup H_{j'}}$-choosable by Lemma \ref{lem:ERT}. 
\end{proof}

\begin{lemma}
\label{lem:deg-5-weak-part}
    Let $H$ be a graph with a function $f \colon V(H) \rightarrow \mathbb{Z}$, and let $\mathcal{H}$ be a subgraph partition of $H$. A part $H_i \in \mathcal{H}$ is an $f$-weak part if the following holds:
    \begin{itemize}
        \item $H_i$ is obtained from a star $K_{1, t}$ with center $x$ and a set $R$ of $t \in \{0,1,2,3\}$ leaves by possibly adding edges between distinct leaves;


        \item $f(x) \geq \deg_H(x) - 1$;

        \item $f(r) \geq \deg_{H}(r) + 1$ for each $r \in R$;
        \item $x$ has a neighbor in at least $5 - t$ distinct parts of $\mathcal{H} - H_i$. 
    \end{itemize}
\end{lemma}
\begin{proof}
    We begin by observing that $\deg_H(x) \geq 5$, so $f(x) \geq 4$. Now, we show that (FIX') holds for each $f$-assignment $L$ on $H_i$. 
    For each $u \in V(H_i)$ and $c \in L(u)$, we may fix $c$ at $u$. Then, if $u \neq x$, we may color $x$ with an available color in $L(x)$, as $L(x)$ has at least $3$ available colors. Subsequently, because 
    $f(r) \geq \deg_H(r) + 1 \geq \deg_{H_i}(r) + 1$ 
    for each $r \in R - u$, Lemma $\ref{lem:ERT}$ allows us to greedily extend our $L$-coloring to the remainder of $H_i$, proving that (FIX') holds.

    Next, we show that (FORB-2) holds for each $f$-assignment $L$ on $H_i$. Fix $u, v \in V(H_i)$ and $c \in L(u) \cup L(v)$. Let $L'(u) = L(u) - c$, $L'(v) = L(v) - c$, and $L'(w) = L(w)$ for each $w \in V(H_i) \setminus \{u, v\}$. We consider two cases: $t \leq 2$ and $t = 3$. 
    If $t \leq 2$, then $|L'(x)| \geq f(x) - 1 \geq 3 \geq t + 1 = \deg_{H_i}(x) + 1$. Additionally, $|L'(r)| \geq f(r) - 1 \geq \deg_H(r) \geq \deg_{H_i}(r)$ for each $r \in R$. Hence, by Lemma \ref{lem:ERT}, there exists an $L'$-coloring of $H_i$. Now suppose that $t = 3$. 
    Then, there exists $r_0 \in R$ such that $L'(r_0) = L(r_0)$, which means that $|L'(r_0)| = f(r_0) \geq \deg_H(r_0) + 1 \geq \deg_{H_i}(r_0) + 1$. Additionally, $|L'(x)| \geq f(x) - 1 \geq 3 = \deg_{H_i}(x)$, and for each $r \in R$, $|L'(r)| \geq f(r) - 1 \geq \deg_H(r) \geq \deg_{H_i}(r)$. Thus, by Lemma \ref{lem:ERT}, there exists and $L'$-coloring of $H_i$.

    Finally, we show that for all $H_j, H_{j'} \in \mathcal{H} - H_i$ and $f^{H_j \cup H_{j'}}$-assignments $L$ on $H_i$, $H_i$ is $L$-colorable. Since $x$ has a neighbor in at least $5 - t$ distinct parts of $\mathcal{H} - H_i$, we have $\deg_H(x) \geq (3 - t) + |N_H(x) \cap V(H_j \cup H_{j'})| + t =  3 + |N_H(x) \cap V(H_j \cup H_{j'})| $. Therefore,
    \begin{align*}
        f^{H_j \cup H_{j'}}(x) &\geq \deg_H(x) - 1 - |N_H(x) \cap V(H_j \cup H_{j'})| \geq 2. 
    \end{align*}
    Additionally, for each $r \in R$, we have $f^{H_j \cup H_{j'}}(r) \geq \deg_H(r) - |N_H(r) \cap V(H_j \cup H_{j'})| + 1 \geq \deg_{H_i}(r) + 1$. Hence, we may produce and $L$-coloring of $H_i$ by first coloring $x$ and then coloring each $r \in R$.
\end{proof}

\section{Structure of a minimal counterexample to Theorem \ref{thm:mad113}}
\label{sec:reducible}
In this section, we fix a minimal counterexample to Theorem \ref{thm:mad113} and establish some properties of this counterexample. 
For the rest of the paper, we fix the constant $\xi = 2^{-48}$.
We use the key observation that if Theorem \ref{thm:mad113} is false, then some graph $G$ of maximum average degree less than $11/3$ is free of induced $(4,\xi)$-reducible subgraphs. Indeed, suppose that every graph of maximum average degree less than $11/3$ has an induced $(4,\xi)$-reducible subgraph. Then,
given some $G$ of maximum average degree less than $11/3$,
as each induced subgraph of $G$ also has maximum average degree less than $11/3$, each induced subgraph of $G$ also
has an induced $(4,\xi)$-reducible subgraph. Then, Lemma \ref{lem:main-k-red}
implies that $G$ is weighted $\epsilon$-flexibly $4$-choosable for $\epsilon = \frac{2 \xi^3}{4} = 2^{-145}$ and thus is not a counterexample to Theorem \ref{thm:mad113}.
Therefore, the negation of Theorem \ref{thm:mad113} implies the existence of a graph $G$ with maximum average degree less than $11/3$ and no induced $(4,\xi)$-reducible subgraph. Using the ideas of Lemma \ref{lem:main-k-red}, it is easy to see that every minimal counterexample to Theorem \ref{thm:mad113} is free of induced $(4,\xi)$-reducible subgraphs.

Based on the discussion above, we fix a graph $G$ with no induced $(4,\xi)$-reducible subgraph, and we 
deduce some structural properties of $G$. If $H$ is a weakly $4$-reducible subgraph of $G$, then we simply write that $H$ is \emph{weakly reducible}.
Given an induced subgraph $H$ of $G$, we define the function $\ell_H(v) = \ell_{H,4}(v) = 4 - \deg_G(v) + \deg_H(v)$ for each $v \in V(H)$.
We observe that $\ell_H(v) \geq \deg_H(v)$ for each $4$-vertex $v \in V(H)$ and $\ell_H(r) \geq \deg_H(r) + 1$ for each $3$-vertex $r \in V(H)$. (Recall that $v$ is a \emph{$d$-vertex} if $\deg_G(v) = d$.)

If $K \subseteq G$ is a subgraph isomorphic to $K_4$, then we say that $K$ is a \emph{special $K_4$} if $K$ contains two $3$-vertices and two $4$-vertices. We say that the edge $e \in E(K)$ joining the two $4$-vertices in $K$ is an \emph{insulated edge}. We say that every edge in $G$ that is not insulated is \emph{conductive}.

We say that a $4$-vertex $v \in V(G)$ is \emph{stressed} if $v$ has exactly two $3$-neighbors. We say that a $4$-vertex $v \in V(G)$ is \emph{conductive} if 
$v$ has exactly one $3$-neighbor.
We say that a \emph{conductive path} is a path whose internal vertices and edges 
are all conductive. We say that two vertices $u,v \in V(G)$ are \emph{conductively connected} if there exists a conductive path with endpoints $u$ and $v$.
We say that a \emph{conductive triangle} is a $K_3$ subgraph whose vertices are all conductive.
If $v \in V(G)$ is a $4$-vertex with no $3$-neighbor or if $v$ is a $5^+$-vertex, then we say that $v$ is
\emph{insulated}.

\subsection{Local properties of a minimal counterexample}

We establish some local properties of our graph $G$. These properties will later help us establish some global structural properties of $G$, which will ultimately allow us to complete our discharging argument.
In order to give clearer intuition for the upcoming lemmas, we briefly summarize our discharging argument in Section \ref{sec:discharging}. We give each vertex a charge, with vertices of higher degree receiving more charge than vertices of lower degree. Then, we show that after redistributing charge according to certain rules,
each vertex of $G$ has nonnegative charge, implying that $G$ has maximum average degree at least $11/3$, a contradiction. In order to carry out our discharging argument, we need to pay special attention to vertices of small degree (which initially receive the least charge), especially those which are near other vertices of small degree.
In particular, we pay special attention to stressed vertices.

\begin{lemma}
\label{lem:no-weak}
    $G$ has no weakly reducible subgraph on at most $24$ vertices.
\end{lemma}
\begin{proof}
    Suppose that $G$ has a weakly reducible subgraph $H$ on at most $24$ vertices. Then, by definition, $H$ is weakly $(\ell_{H,4}, 4)$-reductive. Then, Lemma \ref{lem:weak-to-strong} tells us that $H$ is $(\ell_{H,4}, 4, 4^{-24})$-reductive. As $\xi = 2^{-48} = 4^{-24}$, $H$ is $(4, \xi)$-reducible, a contradiction.
\end{proof}

\begin{lemma}
\label{lem:mindeg3}
    The minimum degree of $G$ is at least $3$.
\end{lemma}
\begin{proof}
    If $G$ has a $2^-$-vertex $v$, then we consider the subgraph $H=v$. We observe that $\ell_H(v) \geq 2$, and 
    hence $H$ is a weakly reducible subgraph of $G$, contradicting Lemma \ref{lem:no-weak}.
\end{proof}

\begin{lemma}
\label{lem:d-2}
    If $v \in V(G)$ is a vertex of degree $d \in \{3,4,5\}$, then $v$ has at most $(d-2)$ $3$-neighbors.
\end{lemma}
\begin{proof}
Suppose that $G$ has a vertex $v$ of degree $d \in \{3,4,5\}$ with at least $d-1$ neighbors of degree $3$.
We let $H$ be the subgraph of $G$ induced by $v$ along with all of its $3$-neighbors, and we note that $|V(H)| \leq 6$.
We observe that $\ell_H(r) > \deg_H(r)$ for each $3$-vertex  $r \in V(H)$ and that $\ell_H(v) \geq 3$. 

We consider an $\ell_H$-assignment $L$ on $H$.
For each $u \in V(H)$ and  $c \in L(u)$,
we can construct an $L$-coloring $\phi$ of $H$ for which $\phi(u) = c$ by first coloring $u$, then coloring $v$ whenever $u\neq v$, and then greedily extending $\phi$ to all uncolored $3$-vertices. Therefore (FIX') holds. Furthermore, for each pair $u_1, u_2 \in V(H)$ and color $c \in L(u_1) \cup L(u_2)$, 
we can construct an $L$-coloring $\phi$ of $H$ for which $c \not \in \{\phi(u_1),\phi(u_2)\}$ by first coloring $u_1$ and $u_2$, then coloring $v$ whenever $v \not \in \{u_1,u_2\}$, and then greedily extending $\phi$ to all $3$-vertices of $H$. Therefore, (FORB-$2$) holds. 

Thus, $H$ is weakly reducible, contradicting Lemma \ref{lem:no-weak}.
\end{proof}

\begin{lemma}
\label{lem:stressed-special-K4}
    Let $u, v \in V(G)$ be stressed vertices, and suppose that $N(u) \cap N(v)$ contains two distinct $3$-vertices $r_1$ and $r_2$. Then, $G[\{u,v,r_1,r_2\}]$ is a special $K_4$.
\end{lemma}
\begin{proof}
    Let $H = G[\{u,v,r_1,r_2\}]$. As $(u,r_1,v,r_2)$ forms a cycle in $G$, $H$ is $2$-connected. 
    We observe that $\ell_H(r_i) = \deg_H(r_i) + 1$ for $i \in \{1,2\}$, and that $\{r_1,r_2\}$ is a dominating set of $H$.
    Suppose that $H$ is not a clique; then, $H$ is a $4$-cycle or a diamond. 

    We claim that if $T$ is an induced
    Gallai tree subgraph of $H$, then 
    $\ell_H$ is not bad on $T$. Indeed, if $T$ consists of a single vertex $w$, then $\ell_H(w) \geq 2 = \deg_T(w) + 2$. If $T$ is a $K_2$, then some $w \in V(T)$ satisfies $\ell_H(w) = 3 = \deg_T(w) + 2$. If $T$ is a triangle, then each $w \in V(T)$ satisfies $\ell_H(w) = 3 > \deg_T(w)$ or there exists $w \in V(T)$ such that $\ell_H(w) = 4 \geq \deg_T(w) + 2$. If $T$ is a path with three vertices, then either an endpoint $w \in V(T)$ satisfies $\ell_H(w) = 3 = \deg_T(w) + 2$, or  $\ell_H(w) = \deg_T(w) + 1$ for each $w \in V(T)$. This exhausts all cases and shows that $\ell_H$ is not bad on $T$.
    
     Furthermore,
    for each adjacent
    pair $w,x \in V(H)$ other than $r_1, r_2$, each component of $H \setminus \{w,x\}$ contains $r_1$ or $r_2$. Hence, Lemma \ref{lem:fix} 
    implies that $H$ is weakly $(\ell_H,4)$-reductive and hence weakly reducible, contradicting Lemma \ref{lem:no-weak}. 
    Therefore, $H$ is a clique, and hence $H$ is a special $K_4$. 
\end{proof}

\begin{lemma}
\label{lem:far_neighbors}
Let $x$ be a $3$-vertex
with two $4$-neighbors $u$ and $v$. If $u$ and $v$ are joined by an induced conductive path $P$, then every vertex $p \in V(P)$ is adjacent to $x$.
\end{lemma}
\begin{proof}
    Suppose $G$ has a triple $(x,u,v)$ that does not satisfy the lemma. Let $(x,u,v)$ be chosen so that the induced conductive path $P$ joining $u$ and $v$ that violates the lemma is as short as possible.
    If $|V(P)|=2$, then clearly $(x,u,v)$ does not violate the lemma. Therefore, we assume that $|V(P)| \geq 3$.
    We write $P' = P \setminus \{u,v\}$.
    As $(x,u,v)$ is chosen so that $P$ is shortest, and as $(x,u,v)$ contradicts the lemma, no vertex of $P'$ is adjacent to $x$. Let $R$ be the set
    of all $3$-neighbors of $P'$.
    Since the vertices of $P'$ are conductive and not adjacent to $x$, it follows that $|R| \geq 1$.
    We write $H = G[V(P) \cup R \cup \{x\} ]$.

    We consider two cases. 
    First, suppose that $|V(P)| \leq 8$. 
    Then, $|V(H)| \leq 15$. 
    As each vertex of $P$ is conductive or stressed, $R \cup \{x\}$ is a dominating set of $H$. Furthermore, each vertex $r \in R \cup \{x\}$ satisfies $\ell_H(r) \geq \deg_H(r) + 1$.
    It is also straightforward to check that for each induced Gallai tree subgraph $T$ of $H$, $\ell_H$ is not bad on $T$, and we show a proof of this claim in the appendix.

    Finally, we observe that for each
    $s \in V(H)$, each component of $H - s$ contains a vertex of $R \cup \{x\}$. 
    Furthermore, for each adjacent pair $s,t \in V(H)$ for which at least one of $s,t$ belongs to $P$, each component of $H \setminus \{s,t\}$ contains a vertex of $R \cup \{x\}$.
    Therefore, Lemma \ref{lem:fix} implies that $H$ is weakly reducible, contradicting Lemma \ref{lem:no-weak}.

    On the other hand, suppose that $|V(P)| \geq 9$.
    We write $A = H[V(P) \cup \{x\}]$ and observe that $A$ is an induced cycle in $H$.
      As $P$ is chosen to be shortest,
    for each $r \in R$, $P[N_P(r)]$ is a subpath of $P$.
    For each $r \in R$, we let $H_r = H[ N_{P'}[r]]$,
    and we observe that $H_r \cap V(P')$
    is a connected vertex subset of $A$ with at most three vertices.
    Furthermore, as each vertex of $P'$ is conductive, it follows that $V(H_r) \cap V(H_{r'}) = \emptyset$ for each distinct pair $r,r' \in R$.
     We also write $H_x = H[\{x,u,v\}]$, and as $P$ is an induced path of length at least $8$, $H_x$ is a path subgraph of $A$ on three vertices.
     We write $\mathcal H = \{H_r : r \in R\} \cup \{H_x\}$ and observe that $\mathcal H$ is a subgraph partition of $H$.
     For each $H_i \in \mathcal H$ and $w,z \in V(H_i)$,
     it holds that $\dist_A(w,z) \leq 2$.
    We aim to apply Lemma \ref{lem:orig_partition} to $\mathcal H$ with $b = d = 4$.

    For each $r \in R$,
    the endpoints of $P'[N_{P'}(r)]$ have neighbors which are at a distance of at least $6$ in $A \setminus H_r$ and hence belong to distinct parts of $\mathcal H$. Therefore,
    Lemma \ref{lem:strong-parts} (\ref{item:internal-petal}) implies that $H_r$ is an $\ell_H$-strong part of $\mathcal H$.

    We also note that $H_x$ 
    consists of a path
    $(u,x,v)$ for which $\ell_H(u) \geq \deg_H(u)$ and $\ell_H(v) \geq \deg_H(v)$,
    and $\ell_H(x) \geq \deg_H(x) + 1$.
    As $|V(A)| \geq 10$,
    the neighbors of $u$ and $v$ have a distance of at least $6$ in $A \setminus H_x$ and hence belong to distinct parts of $\mathcal H$. Therefore, 
    Lemma \ref{lem:weak-parts} (\ref{item:weak-path}) implies that $H_x$ is an $\ell_H$-weak part of $\mathcal H$.
    Thus, we apply Lemma \ref{lem:orig_partition}
    with $b=4$ and $d=4$ to conclude that
    $H$ is 
    $(4,(16)^{-4}4^{-8})$-reducible. As $(16)^{-4} 4^{-8} > \xi$, $H$ is $(4,\xi)$-reducible, 
    a contradiction. Therefore, the lemma holds.
\end{proof}

\begin{lemma}
\label{lem:cycle}
    If $C$ is an induced cycle in $G$ consisting of conductive vertices, then
    $C$ is a triangle, and
        there exists a single $3$-vertex $r \in V(G)$ adjacent to all vertices of $C$. 
\end{lemma}
\begin{proof} 
    Let $R$ be the set of $3$-vertices adjacent to $C$. 
    By Lemma \ref{lem:far_neighbors}, for each $r \in R$, the neighbors of $r$ in $C$ are consecutive.
    Furthermore, $|R| \geq \left \lceil \frac{|V(C)|}{3} \right  \rceil$.
    If $|R| = 1$, then $C$ is a triangle, and we are done. Therefore, we assume that $|R| \geq 2$ and aim for a contradiction.    
    We write $H = G[V(C) \cup R]$, and we claim that $H$ is $(4,\xi)$-reducible, which will give us our contradiction.

    First, we show that $|V(C)| \leq 6$. To this end, suppose that $|V(C)| \geq 7$.
    For each $r \in R$, we define an induced subgraph
    $H_r = H[N_C[r]]$, and we define $\mathcal H = \{H_r: r \in R\}$.
    As each vertex in $C$ is conductive, $\mathcal H$ is a subgraph partition of $H$.
    We observe that each part $H_r \in \mathcal H$ intersects $C$ in a path of at most three vertices, and thus any two vertices in a common part $H_r$ have distance at most two in $C$. Also, for each part $H_r \in \mathcal H$, the two vertices in $C$ adjacent to $H_r$ have distance at least $3$ in $C \setminus H_r$ and therefore belong to distinct parts of $\mathcal H$. Hence, Lemma \ref{lem:strong-parts} (\ref{item:internal-petal}) implies that each $H_r \in \mathcal H$ is an $\ell_H$-strong part.
     Then, by applying Lemma \ref{lem:orig_partition} with $b = d = 4$,
    $H$ is a $(4,(16)^{-4} 4^{-8})$-reducible subgraph of $G$, a contradiction. 
    Hence, we assume that 
    $|V(C)| \leq 6$.

    Now, as each vertex $v \in V(C)$ is conductive, $|R| \leq 6$, and hence $|V(H)| \leq 12$.
    We show that
    $H$ is weakly reducible, contradicting Lemma \ref{lem:no-weak}.
     We note that for each $r \in R$, $\ell_H(r) = \deg_H(r) + 1.$
    Furthermore, as each vertex in $C$ is conductive, $R$ is a dominating set of $H$. 
    
    It is straightforward
    to check that for each induced Gallai tree subgraph $T$ of $H$,
    $\ell_H$ is not bad on $T$, and we show a complete argument in the appendix.
    Furthermore, 
    for each $v \in V(H)$, each component of $H - v$ contains at least one vertex in $R$. Additionally,
    for each adjacent pair $u,v \in V(H)$ with at least one vertex in $C$, each component of $H \setminus \{u,v\}$ contains a vertex of $R$. Therefore, Lemma \ref{lem:fix} implies that $H$ is weakly reducible, a contradiction. 

    Therefore, $|R| = 1$, $C$ is a triangle, and the proof is complete.
\end{proof}

\begin{lemma}
\label{lem:stressed-triangle}
    No stressed vertex has two neighbors in a conductive triangle.
\end{lemma}
\begin{proof}
    Suppose that $v \in V(G)$ is a stressed vertex with two neighbors in a conductive triangle $C$. By Lemma \ref{lem:cycle}, $G$ has a $3$-vertex $r$ satisfying $N(r) = V(C)$. We write $R$ for the set of $3$-neighbors of $v$, along with $r$. We write $H = G[V(C) \cup \{v\} \cup R]$, and we observe that $R$ is a dominating set of $H$. It is easy to check that for each induced Gallai tree subgraph $T$ of $H$, $\ell_H$ is not bad on $T$. We show a complete argument in the appendix. Furthermore, for each $x \in V(H)$,  each component of
    $H-x$ contains a vertex of $R$, and for each adjacent vertex pair $w,x \in V(H)$, each component of $H \setminus \{w,x\}$ contains a vertex of $R$. Therefore, $H$ is weakly reducible by Lemma \ref{lem:fix}, contradicting Lemma \ref{lem:no-weak}.
\end{proof}

\begin{lemma}
\label{lem:triangle-block}
    If $B$ is a $2$-connected subgraph of $G$ consisting of conductive vertices, then $B$ is a triangle.
\end{lemma}
\begin{proof}
    Let $B$ be a $2$-connected subgraph of $G$ consisting of conductive vertices. As $B$ is $2$-connected, $B$ has an induced cycle $C$. By Lemma \ref{lem:cycle}, $C$ is a triangle, and there exists a $3$-vertex $r \in V(G)$ for which $N(r) = V(C)$. We write $C = (u,v,w)$.
    
    Now, suppose that $|V(B)| \geq 4$. Without loss of generality, $u$ has a neighbor $x \in V(B \setminus C)$. 
    By Menger's theorem, there exist edge-disjoint paths $P_1$ and $P_2$ in $B$ from $x$ to $u$. At most one of these paths contains the edge $ux$, so we assume without loss of generality that $P_1 = (u,x)$ and $ux \not \in E(P_2)$. Hence, $B$ has a cycle containing the edge $ux$. 
    Now, let $C'$ be the shortest cycle in $B$ containing $ux$. $C'$ is induced, so by Lemma \ref{lem:cycle}, $C'$ is a triangle, and hence there exists a $3$-vertex $r'$ for which $N(r') = V(C')$. As $C \neq C'$, it follows that $r \neq r'$. Therefore, $r,r' \in N(u)$, and so $u$ is stressed and not conductive, a contradiction. Therefore, $|V(B)| = 3$, and $B$ is a triangle.
\end{proof}

\begin{lemma}
\label{lem:close-stressed}
    Let $P$ be a $4^-$-path in $G$ that joins two stressed vertices $s$ and $t$, such that $P$ does not consist solely of an insulated edge. Suppose that all internal vertices of $P$ are conductive $4$-vertices except for at most one vertex, which we call $x$. Additionally, suppose that either $x$ is either an insulated $4$-vertex, or $x$ is $3$-vertex satisfying
    $x \notin N(s) \cup N(t)$. Then, $|V(P)| \geq 9$.
\end{lemma}
\begin{proof}
    Suppose for contradiction that the lemma is false, so that there exists a counterexample $P$ with $|V(P)| \leq 8$. 
    Let $P$ be a shortest counterexample.
    We let $A = G[V(P)]$. By the minimality of $P$, $A$ is either an induced conductive path or an induced cycle containing exactly one insulated edge $st$. Let $R$ be the set of $3$-neighbors of $A$ not including $V(A)$, and let $H = G[V(A) \cup R]$. Since each vertex in $A$ has at most two $3$-neighbors by Lemma \ref{lem:d-2}, $|V(H)| \leq 24$. We aim to show that $H$ is weakly reducible, contradicting Lemma \ref{lem:no-weak}. To this end, let $L$ be an $\ell_H$-assignment on $H$. We use the fact that $\ell_H(v) \geq \deg_H(v)$ for each $v \in V(A)$, and $\ell_H(r) = \deg_H(r) + 1$ for each $r \in R$.

    To show that (FIX') holds, we observe that for each $w \in V(H)$, each component of $H - w$ contains a 
    vertex $r \in R$, which satisfies $\ell_{H-w}(r) > \deg_{H-w}(r)$. Therefore, by Lemma \ref{lem:ERT}, $H-w$ is $\ell_{H-w}$-choosable, and hence for each $w \in V(H)$ and $c \in L(w)$, $H$ has an $L$-coloring $\phi$ for which $\phi(w) = c$.

    Next, we show that (FORB-$2$) holds. To this end, let $w, z \in V(H)$, and let $c \in L(w) \cup L(z)$. We aim to show that $H$ has an $L$-coloring $\phi$ for which $c \not \in \{\phi(w) , \phi(z)\}$. As (FIX') holds, we may assume that $w \neq z$.
    We observe that for each $v \in V(H)$, $|L(v)| \geq 2$. Additionally, for each $v \in V(A)$, if $v$ has a neighbor in $R$, then $|L(v)| > \deg_A(v)$.

    \begin{enumerate}
        \item 
        \label{item:w-z-in-A}
        First, we consider the case where $w, z \in V(A)$. Let $L'(w) = L(w) - c$, $L'(z) = L(z) - c$, and $L'(v) = L(v)$ for each $v \in V(H) \setminus \{w, z\}$. As $s$ has two $3$-neighbors, neither of which are in $A$, $|L'(s)| \geq \deg_A(s) + 2 - 1 = \deg_A(s) + 1$. Similar reasoning shows that $|L'(t)| \geq \deg_A(t) + 1$. Now, we will consider three subcases: $|L(w)| > \deg_A(w)$ and $|L(z)| > \deg_A(z)$, $|L(w)| = \deg_A(w)$ and $|L(z)| > \deg_A(z)$, and $|L(w)| = \deg_A(w)$ and $|L(z)| = \deg_A(z)$. 

        If $|L(w)| > \deg_A(w)$ and $|L(z)| > \deg_A(z)$, then for each $v \in V(A)$, $|L'(v)| \geq \deg_A(v)$ and $|L'(s)| > \deg_A(s)$. Hence, Lemma \ref{lem:ERT} implies that we may give $A$ an $L'$-coloring. By greedily extending $\phi$ to $R$, we can complete an $L$-coloring of $H$ for which $c \notin \{\phi(w), \phi(z)\}$.

        If $|L(w)| = \deg_A(w)$ and $|L(z)| > \deg_A(z)$, then we first let $\phi(w) = a$ for some $a \in L'(w)$. Then, for each $v \in N_H(w)$, we define $L''(v) = L'(v) - a$, and for each $v \in V(H) \setminus N_H[w]$, we define $L''(v) = L'(v)$. Since $|L(z)| > \deg_A(z)$, we have $|L''(v)| \geq \deg_{A - w}(v)$ for each $v \in V(A - w)$. Additionally, for each $u \in \{s, t\}$, we have $|L''(u)| > \deg_{A - w}(u)$, so each component of $A - w$ has a vertex with strictly more available colors than its degree in $A - w$. It follows from Lemma \ref{lem:ERT} that we we may extend $\phi$ to $A - w$ while avoiding $c$ at $z$. Then, we greedily extend $\phi$ to $R$ to complete an $L$-coloring of $H$ for which $c \notin \{\phi(w), \phi(z)\}$. 

        Finally, we consider the case that $|L(w)| = \deg_A(w)$ and $|L(z)| = \deg_A(z)$. We observe that $\{w, z\} \cap \{s, t\} = \emptyset$. We decompose $A \setminus \{w, z\}$ into at most three conductive paths: one that contains $s$, one that contains $t$, and possibly one that contains neither $s$ nor $t$. We call these paths $P_s$, $P_t$, and $P_{wz}$, respectively. Additionally, we see that neither $w$ nor $z$ is a $3$-vertex. Hence, since $A$ contains only $3$-vertices, conductive $4$-vertices, stressed $4$-vertices, and at most one insulated $4$-vertex, 
        we may assume without loss of generality that $w$ is a conductive $4$-vertex. Let $r$ be the $3$-neighbor of $w$. Then, $r \in V(A)$ since $|L(w)| = \deg_A(w)$. The existence of $r$ implies that $z$ is not insulated, so $z$ is a conductive $4$-vertex. Additionally, as the $3$-neighbor of $z$ is in $A$, $z$ must be adjacent to $r$. Now, we observe that $r \in V(P_{wz})$; otherwise $A$ is not induced. In particular, $A[\{w, r, z\}]$ induces a subpath of $A$, so $w$ and $z$ are not adjacent. We have $|L'(w)| \geq 1$ and $|L'(z)| \geq 1$. Since $w$ and $z$ are not adjacent, we can assign $\phi(w) \in L'(w)$ and $\phi(z) \in L'(z)$. Further, as $r$ is a $3$-vertex, after coloring $w$ and $z$, we may give $r$ a color from $L'(r) \setminus \{\phi(w), \phi(z)\} \neq \emptyset$. Now, for each vertex $v \in V(H) \setminus \{w, r, z\}$, let $L''(v) \subseteq L(v)$ be the set of colors available at $v$ after $w, r$, and $z$ have been colored. For each $v \in V(A) \setminus \{w, r, z\}$, we have $|L''(v)| \geq \deg_{A - \{w, r, z\}}(v)$, and we have $|L''(s)| \geq \deg_{A - \{w, r, z\}}(s) + 1$ and $|L''(t)| \geq \deg_{A - \{w, r, z\}}(t) + 1$. Thus, using Lemma \ref{lem:ERT}, we can extend $\phi$ to $P_s$ and $P_t$. Then, we greedily extend $\phi$ to $R$ to complete an $L$-coloring of $H$ for which $c \notin \{\phi(w), \phi(z)\}$. 

        \item 
        Next, suppose that $w$ and $z$ are both $3$-vertices.
        We define $L'(v) = L(v) - c$ for $v \in \{w,z\}$, and we define $L'(v) = L(v)$ for all $v \in V(H) \setminus \{w,z\}$.
        Then, for each vertex $v \in V(H)$,
        $|L'(v)| \geq \deg_H(v)$.
        If $H$ has a $3$-vertex $r \in R \setminus \{w,z\}$, then as $\ell_H(r) > \deg_H(r)$, Lemma \ref{lem:ERT} implies that $H$ has an $L$-coloring $\phi$ such that $c \not \in \{\phi(x), \phi(y)\}$. Otherwise, $s$ and $t$ have the same set of $3$-neighbors (namely $\{w,z\}$),
        so Lemma \ref{lem:stressed-special-K4} implies that $s$ and $t$ belong to a special $K_4$, which we call $K$. Then, $H[V(K) \cup V(A)]$ is $2$-connected and neither a clique nor odd cycle, so Lemma \ref{lem:ERT} implies that $H$ has an $L$-coloring $\phi$ for which $c \not \in \{\phi(w), \phi(z)\}$. 

        \item 
        Finally, suppose that $w \in V(A)$, $z \in R$, and $w$ is a $4$-vertex.
        We color $z$ with a color $\phi(z) \in L(z) - c$.
        Then, if $|L(w)| \geq 3$, 
        we color $w$ with a color $\phi(w) \in L(w) \setminus \{c,\phi(z)\}$.
        Otherwise, $|L(w)| = 2$, so $w \notin \{s, t\}$, which guarantees $\deg_A(w) = 2$. 
        Because $\deg_A(w) = 2 = |L(w)|$, $w$ has no neighbors in $R$; in particular $wz \not \in E(G)$. Thus, we color $w$ with a color $\phi(w) \in L(w) - c$.

        Now, for each $v \in V(A) - w$, we define $L'(v) \subseteq L(v)$ to be the set of colors $a \in L(v)$ such that $v$ has no neighbor of color $a$. We observe that for each $v \in V(A) - w$, $|L'(v)| \geq \deg_{A-w} (v)$. Furthermore, each component of $A - w$ contains a stressed endpoint $u \in \{s, t\}$. As we demonstrated in (\ref{item:w-z-in-A}), $|L(u)| \geq \deg_{A}(u) + 2$, so $|L'(u)| \geq \deg_{A - w}(u) + 2 - 1 = \deg_{A - w}(u) + 1$. 
        Hence, we give $A - w$ an $L'$-coloring using Lemma \ref{lem:ERT}. Then, we greedily extend our $L$-coloring to $R - z$, giving an $L$-coloring of $H$ for which $c \not \in \{\phi(w),\phi(z)\}$.
    \end{enumerate}

    Thus, $H$ is weakly reducible, contradicting Lemma \ref{lem:no-weak}. Therefore, $|V(P)| \geq 9$. 
\end{proof}

\subsection{Global properties of a minimal counterexample}
In this subsection, we identify certain global properties satisfied by our counterexample $G$.
Each of the following global properties is proven by showing that a certain type of graph of unbounded size is $(4,\xi)$-reducible and therefore forbidden as a subgraph of $G$. Our ability to consider reducible subgraphs of unbounded size is a direct consequence of our framework in Section \ref{sec:framework}. In our proofs, we make extensive use of our partition lemma (Lemma \ref{lem:orig_partition}).

\begin{lemma}
\label{lem:con-stressed}
    No two stressed vertices of $G$ are conductively connected.
\end{lemma}
\begin{proof}
    Suppose that the lemma is false. Choose a triple $(P,u,v)$ consisting of a
    conductive
    path $P$ with stressed endpoints $u$ and $v$ for which $P$ is as short as possible. 
    By Lemma \ref{lem:close-stressed}, $|V(P)| \geq 9$.
    We write $A = G[V(P)]$.
    By the minimality of $|V(P)|$, 
    $A$ is either a conductive path or a cycle with a single insulated edge $uv$.
    Furthermore, no internal vertex of $P$ is stressed.
    We write $R$ for the set of $3$-neighbors of $A$.
    As each vertex in $A$ has at least one 
    neighbor in $R$ and each vertex in $R$ has at most $3$ neighbors in $A$, $|R| \geq \frac{1}{3}|V(A)| \geq 3$.
    We let $H = G[V(A) \cup R]$, and we 
    observe that each vertex in $H$ is a $4^-$-vertex. We
    will apply Lemma \ref{lem:orig_partition} to $H$ with $d = 4$ and $b = 5$.

    We aim to construct a subgraph partition of $H$.
    We write $R'$ for the set of vertices
    $r \in R$ which are adjacent to neither $u$ nor $v$.
    For each vertex $r \in R'$,
    we define the subgraph
    $H_r = H[N_A[r]]$.
    By Lemma \ref{lem:far_neighbors}, $N_A(r)$ is a connected subpath of $P$, and as $r$ is a $3$-vertex, $|N_A(r)| \leq 3$.
    Then, we consider two cases.
    \begin{enumerate}
        \item Suppose that some vertex $r \in R$ is adjacent to both $u$ and $v$. As $|V(A)| \geq 9$, $r$ is not adjacent to all of $V(A)$; therefore, Lemma \ref{lem:far_neighbors} implies that $P$ is not induced. Hence, as observed previously, $A$ is an induced cycle with a single insulated edge $uv$, and thus $u$ and $v$ belong to a special $K_4$, which we call $H_K$. 
        By Lemma \ref{lem:d-2}, each of $u$ and $v$ has at most two $3$-neighbors, 
        so each $3$-neighbor of $u$ and $v$ belongs to $H_K$, and thus $r \in V(H_K)$.
        \item On the other hand, suppose that no vertex $r \in R$ is adjacent to both $u$ and $v$. We write $R_u = N(u) \cap R$ and $R_v = N(v) \cap R$. Then, we define subgraphs $H_u = N_P[R_u]$ and $H_v = N_P[R_v]$.
    \end{enumerate}

    We define a subgraph partition $\mathcal H$ of $H$ depending on the case that we follow above. If we follow Case (1), then we define $\mathcal H = \{H_r: r \in R'\} \cup \{H_K\}$. If we follow Case (2), then we define $\mathcal H = \{H_r: r \in R'\} \cup \{H_u, H_v\}$. As each internal vertex of $P$ is conductive, the parts of $\mathcal H$ are disjoint.

    We claim that for each $H' \in \mathcal H$, $A[V(A) \cap V(H')]$ is a path of at most three vertices. If $H' = H_r$ for some $r \in R'$, then this statement holds by Lemma \ref{lem:far_neighbors}.
    Next, if $H' = H_K$ as in Case (1) above, then $A$ is a cycle, and $A[V(A) \cap V(H_K)]$ is the path $(u,v)$.
    Finally, suppose that we follow Case (2) above and that $H' = H_u$.
    In Case (2), $A$ is a path, so that $A = P$.
    Let $a \in V(A)$ be the vertex of $N_P(R_u)$ whose distance in $P$ from $u$ is greatest, and write $P_{au}$ for the subpath of $P$ with endpoints $a$ and $u$. As some $3$-vertex $r \in R_u$ has both $u$ and $a$ as neighbors, 
    and as all vertices of $P_{au}$ are conductive, it follows from Lemma \ref{lem:far_neighbors} that
    the vertices of $N_P(R_u)$ form the vertex set $V(P_{au})$.
    Furthermore, as $r$ is a $3$-vertex, $|V(P_{au})| \leq 3$.
    Therefore, $A[V(A) \cap V(H_u)]$ is a path with at most $3$ vertices.
    Symmetrically, $A[V(A) \cap V(H_v)]$ is a path with at most $3$ vertices.
    Thus, each part $H' \in \mathcal H$ intersects $A$ in a path of at most $3$ vertices. Hence, for each $H' \in \mathcal H$ and pair $w,w' \in V(H')$,
    $\dist_A(w,w') \leq 2$.

    We claim that each part of $\mathcal H$ is an $\ell_H$-strong part on at most five vertices.
    For each $r \in R'$,
    as $H_r$ intersects $A$ in a path of at most $3$ vertices, the endpoints of $A[V(H_r) \cap V(A)]$ have two neighbors in $A$ whose distance in $A \setminus H_r$ is at least $5$. Hence, the endpoints of $A[V(H_r) \cap V(A)]$ have neighbors in two distinct parts of $\mathcal H -H_r$, and thus    
    $H_r$ is an $\ell_H$-strong part by Lemma \ref{lem:strong-parts} (\ref{item:internal-petal}). Additionally, $|V(H_r)| = 1 + \deg_A(r) \leq 4$.
    In Case (1), $H_K$ is an  $\ell_H$-strong part on four vertices by Lemma \ref{lem:strong-parts} (\ref{item:special-K4}), and $|V(H_K)| = 4$.
    In Case (2), $H_u$ and $H_v$ are $\ell_H$-strong parts by Lemma \ref{lem:strong-parts} (\ref{item:strong-end-piece}). Furthermore, for some $r \in V(H_u)$,
    $|V(H_u)| \leq 1 + |N_A[r]| = 2 + |N_A(r)| \leq 5$;
    similarly, $|V(H_v)| \leq 5$.
    Therefore, each part of $\mathcal H$ is an $\ell_H$-strong part on at most five vertices.
    
    By applying Lemma \ref{lem:orig_partition} with $b = 5$ and $d = 4$, $H$ is $(4,(20)^{-4}4^{-10})$-reducible.
    As  $\xi < 20^{-4}4^{-10}$, $H$ is an induced $(4,\xi)$-reducible subgraph of $G$,
    a contradiction. Thus, the lemma holds.
\end{proof}

\begin{figure}

\begin{center}
\begin{tikzpicture}
[scale=1.2,auto=left,every node/.style={circle,fill=gray!30,minimum size = 6pt,inner sep=0pt}]

\draw [draw=gray!40,fill=gray!40,very thick] (-0.25,-1.25) rectangle (3.25,1.25);
\draw [draw=gray!40,fill=gray!40,very thick] (3.75,-1.25) rectangle (4.25,1.25);
\draw [draw=gray!40,fill=gray!40,very thick] (4.75,-1.25) rectangle (9.25,1.25);

\node(z) at (1.5,1.8) [draw=white,fill=white,minimum size=0pt,inner sep=0pt] {$A$};
\node(z) at (4,1.8) [draw=white,fill=white,minimum size=0pt,inner sep=0pt] {$B$};
\node(z) at (7,1.8) [draw=white,fill=white,minimum size=0pt,inner sep=0pt] {$C$};

\node(z) at (-.75,1) [draw=white,fill=white,minimum size=0pt,inner sep=0pt] {$r$};
\node(z) at (-.75,-1) [draw=white,fill=white,minimum size=0pt,inner sep=0pt] {$s$};
\node(z) at (9.75,1) [draw=white,fill=white,minimum size=0pt,inner sep=0pt] {$t$};
\node(z) at (9.75,-1) [draw=white,fill=white,minimum size=0pt,inner sep=0pt] {$u$};

\node(p1) at (0,0) [draw = black] {};
\node(p2) at (1,0) [draw = black] {};
\node(p3) at (2,0) [draw = black] {};
\node(p4) at (3,0) [draw = black] {};
\node(p5) at (4,0) [draw = black] {};
\node(p6) at (5,0) [draw = black] {};
\node(p7) at (6,0) [draw = black] {};
\node(p8) at (7,0) [draw = black] {};
\node(p9) at (8,0) [draw = black] {};
\node(p10) at (9,0) [draw = black] {};

\node(r1) at (-0.5,1) [draw = black,fill=black] {};
\node(r2) at (-0.5,-1) [draw = black,fill=black] {};
\node(r3) at (1.5,-0.5) [draw = black,fill=black] {};
\node(r4) at (7,1) [draw = black,fill=black] {};
\node(r5) at (9.5,1) [draw = black,fill=black] {};
\node(r6) at (9.5,-1) [draw = black,fill=black] {};

\draw [-] (r5) to [out=135,in=90,looseness=0.4] (p4) {};
\draw [-] (r2) to [out=0,in=270,looseness=0.5] (p6) {};
\draw [-] (r2) to [out=0,in=270,looseness=0.5] (p7) {};
\foreach \from/\to in {p1/p2,p2/p3,p3/p4,p4/p5,p5/p6,p6/p7,p7/p8,p8/p9,p9/p10,r3/p2,r3/p3,r1/p1,r2/p1,r4/p8,p10/r5,p10/r6,r1/r4,r5/p9}
    \draw (\from) -- (\to);
\end{tikzpicture} \\
\begin{tikzpicture}
[scale=1.2,auto=left,every node/.style={circle,fill=gray!30,minimum size = 6pt,inner sep=0pt}]

\node(z) at (4.35,0.5) [draw=white,fill=white,minimum size=0pt,inner sep=0pt] {$s_B$};
\node(z) at (4.35,1) [draw=white,fill=white,minimum size=0pt,inner sep=0pt] {$r_B$};
\node(z) at (4.35,-0.5) [draw=white,fill=white,minimum size=0pt,inner sep=0pt] {$t_B$};
\node(z) at (4.35,-1) [draw=white,fill=white,minimum size=0pt,inner sep=0pt] {$u_B$};


\node(z) at (1.5,2.7) [draw=white,fill=white,minimum size=0pt,inner sep=0pt] {};

\node(z) at (2.75,0.75) [draw=white,fill=white,minimum size=0pt,inner sep=0pt] {$t_A$};
\node(z) at (2.75,-0.75) [draw=white,fill=white,minimum size=0pt,inner sep=0pt] {$u_A$};
\node(z) at (-0.8,1) [draw=white,fill=white,minimum size=0pt,inner sep=0pt] {$r_A$};
\node(z) at (6.2,0.8) [draw=white,fill=white,minimum size=0pt,inner sep=0pt] {$r_C$};
\node(z) at (5.9,-0.9) [draw=white,fill=white,minimum size=0pt,inner sep=0pt] {$s_C$};
\node(z) at (-0.8,-1) [draw=white,fill=white,minimum size=0pt,inner sep=0pt] {$s_A$};

\node(z) at (9.8,1) [draw=white,fill=white,minimum size=0pt,inner sep=0pt] {$t_C$};
\node(z) at (9.8,-1) [draw=white,fill=white,minimum size=0pt,inner sep=0pt] {$u_C$};

\node(p1) at (0,0) [draw = black] {};
\node(p2) at (1,0) [draw = black] {};
\node(p3) at (2,0) [draw = black] {};
\node(p4) at (3,0) [draw = black] {};
\node(p5) at (4,0) [draw = black] {};
\node(p6) at (5,0) [draw = black] {};
\node(p7) at (6,0) [draw = black] {};
\node(p8) at (7,0) [draw = black] {};
\node(p9) at (8,0) [draw = black] {};
\node(p10) at (9,0) [draw = black] {};

\node(rA) at (-.5,1) [draw = black,fill=black] {};
\node(rB) at (4,0.5) [draw = black,fill=black] {};
\node(sB) at (4,1) [draw = black,fill=black] {};
\node(tB) at (4,-0.5) [draw = black,fill=black] {};
\node(uB) at (4,-1) [draw = black,fill=black] {};
\node(rC) at (6.5,1) [draw = black,fill=black] {};
\node(r2) at (-0.5,-1) [draw = black,fill=black] {};
\node(sC) at (5.5,-1) [draw = black,fill=black] {};
\node(r3) at (1.5,-0.5) [draw = black,fill=black] {};
\node(r4) at (7,1) [draw = black,fill=black] {};
\node(tC) at (9.5,1) [draw = black,fill=black] {};
\node(tA) at (3,1) [draw = black,fill=black] {};
\node(uA) at (3,-1) [draw = black,fill=black] {};
\node(r6) at (9.5,-1) [draw = black,fill=black] {};

\foreach \from/\to in {tA/p4,p1/p2,p2/p3,p3/p4,p4/p5,p5/p6,p6/p7,p7/p8,p8/p9,p9/p10,r3/p2,r3/p3,rA/p1,r2/p1,r4/p8,p10/r5,p10/r6,r5/p9,rC/r4,sC/p7,sC/p6}
    \draw (\from) -- (\to);
\end{tikzpicture} 
\end{center}
\caption{The upper figure shows a subgraph $H$ of $G$. The six dark vertices are $3$-vertices, and the light vertices are $4^+$-vertices. We define a set $X = \{r,s,t,u\}$, as well as a subgraph partition $\mathcal B = \{A,B,C\}$ of $H \setminus X$. The lower figure shows the graph $\spl(H,X,\mathcal B)$. In Lemma \ref{lem:at-most-one-stressed}, we also show that $(H,X,\mathcal B)$ satisfies the assumptions of Lemma \ref{lem:crossing-over}.}
\label{fig:split}
\end{figure}
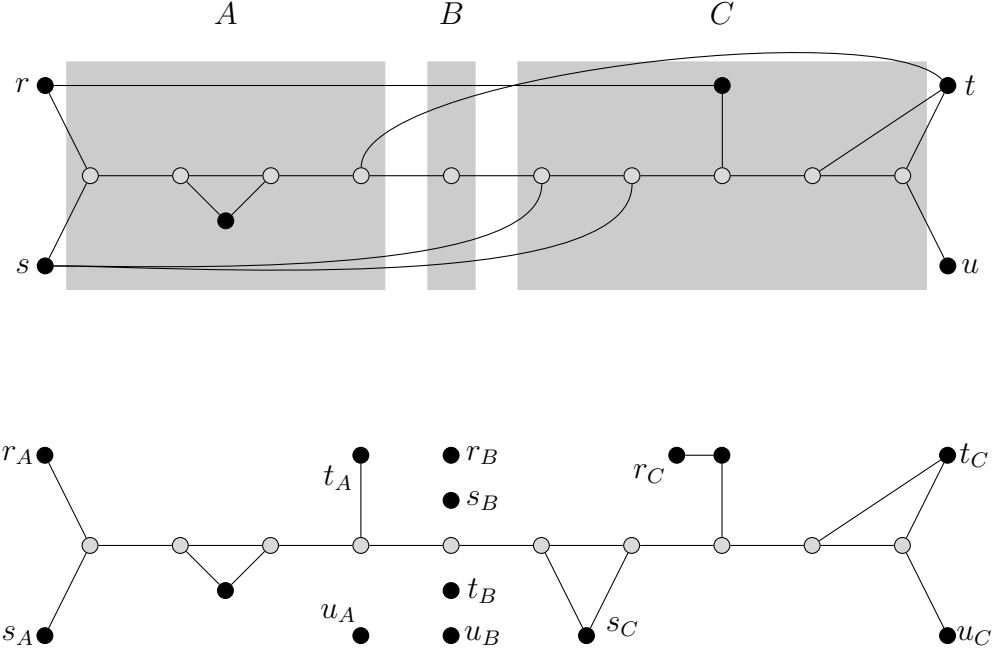

Next, we introduce the notion of a \emph{split graph}, which will help us use Lemma \ref{lem:orig_partition} to rule out more types of subgraphs of $G$.
\begin{definition}
    Let $H$ be an induced subgraph of $G$, let $X \subseteq V(H)$, and let $\mathcal{B}$ be a 
    subgraph partition
    of $H \setminus X$. We define $\spl(H, X, \mathcal{B})$ using the following procedure.
    \begin{enumerate}
        \item Let $H'$ be a copy of $H \setminus X$.
        \item For each $r \in X$ and $B \in \mathcal{B}$, add a vertex $r_B$ to $H'$
        with neighborhood $N_B(r)$.
        \item Define $\spl(H, X, \mathcal{B}) := H'$. 
    \end{enumerate}
    Furthermore, 
    define $\ell_{\spl} \colon V(\spl(H, X, \mathcal{B})) \rightarrow \mathbb Z$ with the following properties: 
    \begin{enumerate}
        \item For each $v \notin X$, $\ell_{\spl}(v) = \ell_H(v)$;

        \item For each $r \in X$ and $B \in \mathcal{B}$, $\ell_{\spl}(r_B) = \ell_H(r)$.
    \end{enumerate}
    Finally, given a list assignment $L:V(H) \rightarrow 2^{\mathbb N}$, we define a list assignment $L_{\spl}:V(\spl(H,X,\mathcal B)) \rightarrow 2^{\mathbb N}$ as follows:
    \begin{enumerate}
        \item For each $v \not \in X$,  $L_{\spl}(v) = L(v)$;
        \item For each $r \in X$ and $B \in \mathcal B$, $L_{\spl}(r_B) = L(r)$.
    \end{enumerate}

\end{definition}

For an example of a graph $\spl(H, X, \mathcal B)$, see Figure \ref{fig:split}.
The following lemma shows that under certain conditions, if $H$ is an induced subgraph of $G$ and $\spl(H,X,\mathcal B)$ is reductive, then $H$ is reducible.

\begin{lemma}
\label{lem:crossing-over}
    Let $H$ be an induced subgraph of $G$. Let $R \subseteq \{r \in V(H) \colon \deg_G(r) = 3\}$,
    and let $m$ and $q$ be fixed positive integers.
    Suppose there exists a subset $X \subseteq R$, as well as a subgraph partition $\mathcal{B}$ of $H \setminus X$ 
    such that the following hold: 
    \begin{enumerate}
        \item \label{item:X-leq-q-and-B-leq-m} $|X| \leq q$, and 
        $|\mathcal{B}| \leq m$

        \item \label{item:r-neighbor}
        Each vertex $r \in R$ has a neighbor in $V(H) \setminus R$;

        \item \label{item:color-branch} For each $r \in X$ and  $B \in \mathcal{B}$, 
        $H \setminus (V(B) \cup \{r\})$ is $\ell_{H \setminus (V(B) \cup \{r\})}$-choosable;

        \item \label{item:r1r2} For each pair $r_1, r_2 \in X$, $H \setminus  \{r_1, r_2\}$ is $\ell_{H \setminus \{r_1, r_2\}}$-choosable. 

        \item \label{item:split-coloring} $\spl(H, X, \mathcal{B})$ is 
        $(\ell_{\spl}, 4, \gamma)$-reductive;
    \end{enumerate}
    Then $H$ is $(4,\alpha)$-reducible for the constant $\alpha = \min\{ \frac{\gamma}{3qm}, 
    \frac{1}{12q^2} \}$.
\end{lemma}
\begin{proof}
    Fix an $\ell_H$-assignment $L$ on $H$.
        We write $S = \spl(H,X,\mathcal B)$.
    We observe that $L_{\spl}$ is an $\ell_{\spl}$-assignment on $S$.
    Therefore, 
    as $S$ is $(\ell_{\spl},4,\gamma)$-reductive, 
    there exists a distribution on $L_{\spl}$-colorings of $S$ that witnesses (FIX) and (FORB) for $k = 4$ with probability constant $\gamma$.

    We utilize the following procedure to obtain a random $L$-coloring $\phi$ of $H$. 
    First, we select a number $\rho \in \{1, 2, 3\}$ uniformly at random.
    Then, we proceed as follows based on the value of $\rho$.

    \begin{enumerate}
    \item If $\rho = 1 $, we randomly sample an $L_{\spl}$-coloring $\psi$ of $S$ using the distribution described above.
    Then, we define $\phi$ on $H$ as follows. First, for each $v \in V(H) \setminus X$, we let $\phi(v) := \psi(v)$. Then, using the fact that 
    each $3$-vertex $r \in V(H)$ satisfies 
    $|L(r)| = \ell_H(r) > \deg_H(r)$, 
    we greedily extend $\phi$ to $X$.

    \item If $\rho = 2$, we again sample an $L_{\spl}$-coloring $\psi$ of $S$ from the distribution described above.
    Then, we select $r \in X$ and $B \in \mathcal{B}$ uniformly at random. We let $\phi(r) := \psi(r_B)$. Afterward, for each $v \in V(B)$, we let $\phi(v): = \psi(v)$. Finally, using property (\ref{item:color-branch}), we extend $\phi$ to $H \setminus (V(B) \cup \{r\})$. 

    \item Finally, if $\rho = 3$, we choose distinct $r_1, r_2 \in X$ uniformly at random. First, we assign $\phi(r_1)$ uniformly at random from $L(r_1)$. Then, we assign $\phi(r_2)$ uniformly at random from the available colors of $L(r_2)$. Finally, using property (\ref{item:r1r2}), we extend $\phi$ to $H \setminus \{r_1, r_2\}$. 
    \end{enumerate}

    We now show that the distribution produced by this random coloring procedure satisfies (FIX) and (FORB) with a probability constant $\alpha = \min\{ \frac{\gamma}{3qm}, 
    \frac{1}{12q^2} \}$.
    
    First, we show that (FIX) is satisfied. Consider a vertex  $v \in V(H)$ and a color $c \in L(v)$. We consider two cases: $v \in X$ and $v \notin X$. First, suppose that $v \in X$. With probability $\frac 13$, $\rho = 3$. Then, with conditional probability at least $\frac 1q$, $v =  r_1$. Then, with further conditional probability at least $\frac 14$, $c$ is chosen for $\phi(r_1)$. Hence, with probability at least $\frac{1}{12q}$, $\phi(v) = c$. On the other hand, suppose that $v \notin X$. Then, with probability $\frac 13$, $\rho = 1$. Then, since $S$ is $(\ell_{\spl}, 4, \gamma)$-reductive,
    $\psi(v) = c$ with probability at least $\gamma$. Hence, $\phi(v) = \psi(v) = c$ with conditional probability $\gamma$. So, with probability at least $ \frac{\gamma}{3}$, $\phi(v) = c$. As $\alpha < \min \{ \frac{1}{12q} , \frac{\gamma}{3} \}$, (FIX) is satisfied with a probability constant $\alpha$.

    Now, we show that (FORB) is satisfied. Fix $u, v \in V(H)$ and $c \in L(u) \cup L(v)$. 
    We consider three cases.
    \begin{enumerate}
    \item 
    First, suppose that $u, v \notin X$. With probability $\frac 13$, $\rho = 1$. Then, because $S$ is  $(\ell_{\spl}, 4, \gamma)$-reductive,
    when we produce a random $L_{\spl}$-coloring $\psi$ of   $S$,
    $\Pr(\psi(u) \neq c \land \psi(v) \neq c) \geq \gamma$. So with conditional probability at least $\gamma$, $\phi(u) = \psi(u) \neq c$ and $\phi(v) = \psi(v) \neq c$. Hence with probability at least $\frac{\gamma}{3}$,  $c \not \in \{\phi(u), \phi(v)\}$.
    \item Next, suppose that
    $u \in X$ and $v \notin X$.
    With probability $\frac 13$, $\rho = 2$. Then, with conditional probability at least $\frac{1}{qm}$, 
    the randomly selected vertex $r \in X$ is $u$, and $v$ is contained in the randomly selected subgraph $B \in \mathcal B$. Because $S$ is $(\ell_{\spl},4,\gamma)$-reductive, 
    when we produce a random $L_{\spl}$-coloring $\psi$ of   $S$,
    $\Pr(\psi(u_B) \neq c \land \psi(v) \neq c) \geq \gamma$. Hence, with further conditional probability at least $\gamma$, $\phi(u) = \psi(u_B) \neq c$ and $\phi(v) = \psi(v) \neq c$. Overall, with probability at least $\frac{\gamma}{3qm}$, $c \not \in \{\phi(u), \phi(v)\}$.
    \item Finally, suppose that $u, v \in X$. With probability $\frac 13$, $\rho = 3$. Then, with conditional probability at least $\frac{1}{q^2}$, $u = r_1$ and $v = r_2$. Since $u$ is adjacent to at least one vertex in $V(H) \setminus R$, $|L(u)| \geq 2$. Similarly, at least two colors are available in $L(v)$ when $v$ is colored. Hence, with further conditional probability at least $\frac 1{2^2} $, $c$ is avoided at $u$ and $v$. Overall, with probability at least $\frac {1}{12q^2}$, $c \not \in \{\phi(u), \phi(v)\}$.
    \end{enumerate}

    Overall, our distribution on $L$-colorings $\phi$ of $H$ satisfies (FIX) and (FORB) for $k = 4$ and a probability constant of $\alpha = \min\{ \frac{\gamma}{3qm}, 
    \frac{1}{12q^2} \}$.
\end{proof}

The next lemma shows that
in a split graph, a single copy of a $3$-vertex can often be strong part in a vertex partition.

\begin{lemma}
\label{lem:hanging-red}
    Let $H$ be an induced subgraph of $G$.
    Let $X \subseteq V(H)$, and let $\mathcal B$ be a subgraph partition of $H \setminus X$.
    Let $S = \spl(H,X,\mathcal B)$.
    Let $r \in X$ 
    and $B \in \mathcal B$. 
    Suppose that $r$
    has a neighbor $w \in V(H) \setminus X$ and that $w \not \in N_S(r_B)$.
    If $\mathcal S$ is a subgraph partition of $S$ and the $K_1$ consisting of $r_B$ is an element of $\mathcal S$, then $r_B$ is an $\ell_{\spl}$-strong part of $\mathcal S$.
\end{lemma}
\begin{proof}
    As $w \not \in N_S(r_B)$, it follows that $\deg_H(r) \geq \deg_S(r_B) + 1$. As $r$ is a $3$-vertex, 
    \[\ell_{\spl}(r_B) = \ell_H(r) = \deg_{H}(r) + 1 \geq  \deg_S(r_B) + 2.
    \]  
    Now, 
    we show that $r_B$ is an $\mathcal S$-strong part as follows.
    \begin{enumerate}
    \item Let $L$ be an $\ell_{\spl}$-assignment on $r_B$.
    As $\ell_{\spl}(r_B) \geq 2$, (FIX') and (FORB-$2$) hold.
    \item Next, let $S_j \in \mathcal S$ be distinct from $r_B$, and let $L$ be an 
    $\ell_{\spl}^{ S_j}$-assignment $L$ on $r_B$.
    As $S_j \subseteq S - r_B$, it holds that
    \[
        \ell_{\spl}^{S_j}(r_B) \geq \ell_{\spl}^{S - r_B}(r_B) = \ell_{\spl}(r_B) - \deg_S(r_B) \geq 2.
    \]
    Thus, (FORB-$1$) holds.
    \item Finally, let 
    $L$ be an $\ell_{\spl}^{S - r_B}$-assignment on $r_B$. As 
    $\ell_{\spl}^{S- r_B}(r_B)
    = 
    \ell_{\spl}(r_B)- \deg_{S}(r_B) \geq 2$, $r_B$ is $L$-colorable. 
    \end{enumerate}
    Therefore,  $r_B$ is an $\ell_{\spl}$-strong part of $\mathcal S$. 
\end{proof}

\begin{lemma}
\label{lem:434-path}
No two stressed vertices $u,v \in V(G)$ are joined by a path of conductive edges whose internal vertex set
consists of conductive vertices and one $3$-vertex $x$ for which $x \not \in N(u) \cup N(v)$.
\end{lemma}
\begin{proof}
 Suppose that the lemma is false. Let $u,v \in V(G)$ be two stressed vertices which are joined by a path $P$ consisting of conductive edges
 whose internal vertex set
 consists of conductive vertices and a
 $3$-vertex $x \not \in N(u) \cup N(v)$.
We take $P$ to be as short as possible, so $A:=G[V(P)]$ is either a path or a cycle with exactly one insulated edge $uv$.
 We write $P' = P-x$, and we write $R$ for the set of $3$-vertices which are adjacent to $P'$, not including $x$.
By Lemma \ref{lem:close-stressed},
$|V(A)| \geq 9$.

    We write $P_1$ and $P_2$ for the two components of $P'$,
    so that $u \in P_1$ and $v \in P_2$. 
    We write $X = N_R(\{u,v\})$, and by Lemma \ref{lem:d-2}, $|X| \leq 4$.
    We claim that no $r \in R \setminus X$ has a neighbor in both $P_1$ and $P_2$.
    To this end, suppose that $r$ has neighbors $p_1$ and $p_2$ respectively in $P_1$ and $P_2$.
    Neither $p_1$ nor $p_2$ is adjacent to $x$, as otherwise $p_1$ or $p_2$ is stressed and thus belongs to $\{u,v\}$.
    Let $P_{up_1}$ be the subpath of $P$ from $u$ to $p_1$, and let $P_{p_2 v}$ be the subpath of $P$ from $p_2$ to $v$. Then, the path $(P_{u p_1}, r, P_{p_2 v})$ is shorter than $P$ and has an internal vertex set consisting of conductive vertices and exactly one $3$-vertex $r \not \in N(u) \cup N(v)$, contradicting the minimality of $P$.

        We write $H = G[V(A) \cup R]$.
        For $i \in \{1,2\}$, we write $B_i = H[N_R[P_i]] \setminus X$.
        We write $\mathcal B = \{B_1, B_2, x\}$ and observe that $\mathcal B$ is a subgraph partition of $H \setminus X$. Finally, we write $S = \spl(H,X,\mathcal B)$. 
    For each vertex $v \in V(H) \setminus X$, we identify the copy of $v$ in $S$ with the vertex $v \in V(H)$. 

    We aim to use Lemma \ref{lem:crossing-over} to show that $S$ is $(4,\alpha)$-reducible. To this end, we check the conditions of Lemma \ref{lem:crossing-over}.
    As $|X| \leq 4$ and $|\mathcal B| = 3$, (\ref{item:X-leq-q-and-B-leq-m}) holds with $q = 4$ and $m = 3$. By construction, each $r \in R$ has a neighbor in $V(H) \setminus R$, so (\ref{item:r-neighbor}) holds.

    Next, we show that (\ref{item:color-branch}) holds.
    For each $B \in \mathcal B$ and $r \in X$,
    write $H' = H \setminus (V(B) \cup \{r\})$.
    Each vertex $v \in V(H')$ is a $4^-$-vertex, so $\ell_{H'}(v) \geq \deg_{H'}(v)$.
    We also observe that each component of $H'$ contains an endpoint of $P$ with at least one $3$-neighbor $r' \in X - r$ satisfying $\ell_{H'}(r') > \deg_{H'}(r')$. Therefore,
    Lemma \ref{lem:ERT} implies that
    $H'$ is $\ell_{H'}$-choosable, and (\ref{item:color-branch}) holds.

    To show that (\ref{item:r1r2}) holds, consider a pair $r_1, r_2 \in X$, and define $H_{\ref{item:r1r2}} = H \setminus \{r_1, r_2\}$.
    We observe that $H_{\ref{item:r1r2}}$ is connected, as each pair of vertices in $H_{\ref{item:r1r2}}$ is joined by a path whose internal vertices belong to $P$.
    As before, $\ell_{H_{\ref{item:r1r2}}}(v) \geq \deg_{H_{\ref{item:r1r2}}}(v)$ for each $v \in V(H_{\ref{item:r1r2}})$, and furthermore $\ell_{H_{\ref{item:r1r2}}}(x) > \deg_{H_{\ref{item:r1r2}}}(x)$. Therefore,
    Lemma \ref{lem:ERT} implies that $H_{\ref{item:r1r2}}$ is $\ell_{H_{\ref{item:r1r2}}}$-choosable.

    We finally aim to show that (\ref{item:split-coloring}) holds.
    To this end, we 
    give $S = \spl(H,X,\mathcal B)$ a subgraph partition $\mathcal S$ in which all parts are $\ell_{\spl}$-weak, all parts but one are $\ell_{\spl}$-strong, and each part has at most $5$ vertices.

    Let $R' = (R \setminus X) \cup \{r_B: r \in X, B \in \mathcal B\}$; that is,
    $R'$ is the set of vertices in $S$ which correspond to $3$-vertices in $H$, excluding $x$. 
    For each $r \in R'$ which is not adjacent to $u$ or $v$, we define a subgraph $S_r = S[N_{P'}[r]]$. 
    We also define $R_u = N_S(u) \cap R'$, and we define $R_v = N_S(v) \cap R'$.
    We further define subgraphs $S_u = S[N_{P'}[R_u]]$ and $S_v = S[N_{P'}[R_v]]$.
    We further define a part $S_x = S[N_P[x]]$, and we observe that $S_x$ is a path on three vertices.
    We write 
    $\mathcal{S} = \{S_r \colon r \in R' \setminus (R_u \cup R_v)\} \cup \{S_u, S_v, S_x\}$, and we observe that $\mathcal S$ is a subgraph partition of $S$.
    By construction, no part of $\mathcal S \setminus \{S_x\}$ has a vertex in both $P_1$ and $P_2$.
    Hence, Lemma \ref{lem:far_neighbors} implies that each part of $\mathcal S$ intersects $A$ in a connected subgraph of $P$.
    By a similar argument to that of Lemma \ref{lem:con-stressed}, 
    each part of $\mathcal S$ intersects $A$ in a path of at most $3$ vertices. Therefore, for each part $S' \in \mathcal S$ and pair $w,w' \in V(S')$, $\dist_A(w,w') \leq 2$.

    We write $x_1, x_2$ for the endpoints of $S_x$, so that $x_1 \in V(P_1)$ and $x_2 \in V(P_2)$.
    We observe that $\ell_{\spl}(x_1) = \ell_{\spl}(x_2) = 2$,
    and $\ell_{\spl}(x) \geq  3$. 
    Furthermore, the neighbors of $S_x$ in $A$ have a distance of at least $5$ in $A$ and thus belong to distinct parts of $\mathcal S$.
    Therefore, Lemma 
    \ref{lem:weak-parts} 
    (\ref{item:weak-path}) 
    implies that $S_x$ is an 
    $\ell_{\spl}$-weak part of $\mathcal S$.
    Next, we consider a vertex
    $r \in R' \setminus (R_u \cup R_v)$
    with a neighbor in $P'$. 
    As $r$ corresponds to a $3$-vertex of $G$, 
    $\ell_{\spl}(r) \geq  \deg_H(r) + 1 \geq \deg_S(r) + 1$, and $\ell_{\spl}(v) \geq \deg_H(v) \geq  \deg_S(v)$ for each $v \in N_P(r)$.
    As $A[V(S_r) \cap V(A)]$ is a path of length at most $2$, the neighbors of $S_r$ in $A$ have a distance of at least $5$ in the graph $A \setminus S_r$ and hence belong to distinct parts of $\mathcal S$.
 Hence, $S_r$ is an $\ell_{\spl}$-strong part by Lemma 
 \ref{lem:strong-parts} (\ref{item:internal-petal}).
By a similar argument, $S_u$ and $S_v$ are $\ell_{\spl}$-strong parts of $\mathcal S$ described in Lemma \ref{lem:strong-parts} (\ref{item:strong-end-piece}).
    Finally, for each $r \in R' \setminus (R_u \cup R_v)$
    with no neighbor in $P'$, 
    $r$ is a copy of a vertex $r_0 \in R$ for which there exists a vertex $w \in V(H) \setminus X$ adjacent to $r_0$ and not to $r$. Therefore, Lemma \ref{lem:hanging-red} implies that $S_r = r$ is an $\ell_{\spl}$-strong part of $\mathcal S$.

    By applying Lemma \ref{lem:orig_partition}
    with $b=5$ and $d=4$, $S$ is $(\ell_{\spl},4,20^{-4}4^{-10})$-reductive. Hence, Lemma \ref{lem:crossing-over} implies that $H$ is $(4, 20^{-4}4^{-10}36^{-1})$-reducible.
    As $\xi < 20^{-4}4^{-10}36^{-1}$, we have a contradiction, and the proof is complete.
\end{proof}

\begin{lemma}
\label{lem:stressed-insulated}
    Each stressed vertex in $V(G)$ is conductively connected with an insulated vertex.
\end{lemma}
\begin{proof}
    Let $v \in V(G)$ be a stressed vertex,
    and suppose that $v$ is not conductively connected with an insulated vertex.
    Let $W \subseteq V(G)$ be the set of conductive vertices with which $v$ is conductively connected. 
    As $v$ is a $4$-vertex with two $3$-neighbors, Lemma \ref{lem:con-stressed} implies that $v$ has at least two $4^+$-neighbors, at least one of which,
    say $w$, 
    is joined to $v$ by a conductive edge. 
    If $w$ is insulated, then we are done; otherwise, $w$ is conductive, and $W$ 
    is nonempty.
    
    Note that for each $w \in W$, if $x \in N(w)$ is stressed, then Lemma \ref{lem:con-stressed} implies that $x = v$. 
    If $x \in N(w)$ is a $4^+$-vertex which is not stressed, then as $v$ is not conductively connected with an insulated vertex, $x$ is conductive.
    Therefore, if $w \in N(v)$, then $w$ has two conductive neighbors; otherwise, $w$ has three conductive neighbors.

    Consider a terminal block $B$ of $G[W]$. 
    As $G[W]$ has a minimum degree of at least $2$, $B$ is not isomorphic to $K_2$,
    so $B$ is $2$-connected.
    Thus, by Lemma \ref{lem:triangle-block},
    $B$ is a triangle.
    Furthermore, as $B$ is a terminal block in $G[W]$, two of the vertices $b_1,b_2 \in V(B)$
    satisfy $\deg_W(b_1) = \deg_W(b_2) = 2$, and hence $b_1, b_2 \in N(v)$.
    Thus, the stressed vertex $v$ has two neighbors in a conductive triangle, contradicting Lemma \ref{lem:stressed-triangle}.
\end{proof}

\begin{lemma}
\label{lem:at-most-one-stressed}
    If $x \in V(G)$ is an insulated $4$-vertex,
    then $x$ is conductively connected with at most one stressed vertex. 
\end{lemma}
\begin{proof}
    Suppose that $x$ is conductively connected with two stressed vertices $u$ and $v$. Let $P_1$ and $P_2$ be shortest conductive paths joining $x$ with $u$ and $v$, respectively.
    We write $P = P_1 \cup P_2$, and by Lemma \ref{lem:close-stressed}, $|V(P)| \geq 9$.
    Let $R$ be the set of $3$-neighbors of $P_1 \cup P_2$, and observe that $|R| \geq 4$.
    Let $H = G[V(P) \cup R]$. We aim to show that $H$ is $(4,\xi)$-reducible.

    We claim that $A:=H[V(P)]$ is either a path or a cycle. Indeed, suppose that some vertex $p_1 \in V(P)$ is adjacent to an internal vertex $p_2 \in V(P) \setminus N_P(p_1)$. As $P_1$ and $P_2$ are chosen to be shortest, it follows without loss of generality that 
    $p_1 \in V(P_1)$ and $p_2 \in V(P_2)$. As $p_2$ is an internal vertex of $P$ and hence is not stressed, the edge $p_1 p_2$ is conductive. Therefore, the endpoints of $P$ are conductively connected stressed vertices, contradicting Lemma \ref{lem:con-stressed}. Therefore, $A$ is either a path or cycle.

    We write $X = N_R(\{u,v\})$.
    We claim that no vertex $r \in R \setminus X$ has a neighbor in both $P_1$ and $P_2$.
     Indeed, suppose that $r \in R \setminus X$ has neighbors $p_1$ and $p_2$ respectively in $P_1$ and $P_2$.
    Let $P_{up_1}$ be the subpath of $P$ from $u$ to $p_1$, and let $P_{p_2 v}$ be the subpath of $P$ from $p_2$ to $v$. Then, the path $(P_{u p_1}, r, P_{p_2 v})$ is a path joining $u$ and $v$ whose internal neighbors are conductive, with the exception of a single $3$-vertex $r$ which is adjacent to neither $u$ nor $v$.
    This contradicts Lemma \ref{lem:434-path}.

     For each $r \in R \setminus X$, Lemma \ref{lem:far_neighbors} along with the claim above implies that $P[N_P(r)]$ is a connected subpath of $P$ which does not contain an endpoint of $P$. 
     Similarly, for each $r \in (R \cap N(u)) \setminus N(v)$, $P[N_{P_2}(r)]$ is a connected subpath of $P$ containing no endpoint of $P$. Additionally, for each $r \in (R \cap N(v)) \setminus N(u)$, 
      $P[N_{P_1}(r)]$ is a connected subpath of $P$ containing no endpoint of $P$. 
        Finally, we observe that for each $r \in R \cap N(u)$, $P[N_{P_1}(r)]$ is a subpath of $P$ containing $u$. 
        Similarly, 
        for each $r \in R \cap N(v)$, $P[N_{P_1}(v)]$ is a subpath of $P$ containing $v$.

    Now, for $i \in \{1,2\}$, we write $H_i$ for the subgraph of $H$ induced by $P_i -x$ and the $3$-neighbors of $P_i$, excluding $X$.
    That is, $H_i = H[N_{R\setminus X}[P_i -x]]$.
    We write $\mathcal B = \{H_1,H_2,x\}$, and we observe that $\mathcal B$ is a subgraph partition of $H \setminus X$. We write $S = \spl(H,X,\mathcal B)$.   
    
    We
    use Lemma \ref{lem:crossing-over} to show that $S$ is $(4, \xi)$-reducible.
    As $|X| \leq 4$ and $|\mathcal B| = 3$, (\ref{item:X-leq-q-and-B-leq-m})     holds with $q = 4$ and $m = 3$. By construction, each vertex $r \in R$ has a neighbor in $P =
    V(H) \setminus R$, so (\ref{item:r-neighbor}) holds.

   Next, we show that (\ref{item:color-branch}) holds.
    For each $B \in \mathcal B$ and $r \in X$, 
    write $H' = H \setminus (V(B) \cup \{r\})$.
    As each vertex in $H'$ is a $4^-$-vertex, each
    $v \in V(H')$ satisfies $\ell_{H'}(v) \geq \deg_{H'}(v)$.
    We also observe that each component of $H'$ contains an endpoint of $P$ with at least one $3$-neighbor $r' \in R - r$ satisfying $\ell_{H'}(r') > \deg_{H'}(r')$. Therefore,
    Lemma \ref{lem:ERT} implies that
    $H'$ is $\ell_{H'}$-choosable, and (\ref{item:color-branch}) holds.

     To show that (\ref{item:r1r2}) holds, consider a pair $r_1, r_2 \in X$, and define $H_{\ref{item:r1r2}} = H \setminus \{r_1, r_2\}$.
    We observe that $H_{\ref{item:r1r2}}$ is connected, as each pair of vertices in $H_{\ref{item:r1r2}}$ is joined by a path whose internal vertices belong to $P$.
    As before, $\ell_{H_{\ref{item:r1r2}}}(v) \geq \deg_{H_{\ref{item:r1r2}}}(v)$ for each $v \in V(H_{\ref{item:r1r2}})$.
    Furthermore, as $|R| \geq 4$, $R \setminus \{r_1,r_2\}$ is nonempty,
    so some vertex $r \in R \setminus \{r_1,r_2\}$
    satisfies $\ell_{H_{\ref{item:r1r2}}}(r) > \deg_{H_{\ref{item:r1r2}}}(r)$. Therefore,
    Lemma \ref{lem:ERT} implies that $H_{\ref{item:r1r2}}$ is $\ell_{H_{\ref{item:r1r2}}}$-choosable.

    Finally, we show that (\ref{item:split-coloring}) holds.
    We aim to partition $S$ into $\ell_{\spl}$-weak subgraphs such that all but one of these subgraphs is $\ell_{\spl}$-strong.
    We define the following induced subgraphs of $S$.
    \begin{itemize}
        \item We define  $S_x = x$.
        \item We define $R_u \subseteq N_S(u)$ to be the set of neighbors of $u$ in $S$ of the form $r_B$, where $r \in R$ and $B = H_1$. Then, we define $S_u = S[N_A[R_u]]$.
        \item We define $R_v \subseteq N_S(v)$ to be the set of neighbors of $v$ in $S$ of the form $r_B$, where $r \in R$ and $B = H_2$. Then, we define $S_v = S[N_A[R_v]]$.
        \item  If $r \in R \setminus X$, then
        we define $S_r = S[N_A[r]]$.
       \item If $r \in X$, 
       $B \in \mathcal B$, and $r$ has no neighbor in $B \cap \{u,v\}$,
           then we define $S_{r_B} = S[N_A[r_B]]$.

    \end{itemize}
    As $|V(A)| = |V(P)| \geq 9$,
    either $\dist_P(x,u) \geq 4$ or $\dist_P(x,v) \geq 4$.
    Therefore, Lemma \ref{lem:far_neighbors} implies that 
    $x$ has a conductive neighbor $w$ with a neighbor $r_w \notin N_S(u) \cup N_S(v)$ that corresponds with a $3$-neighbor of $w$ in $H$.
    We write $S_w = S_x \cup S_{r_w}$.  
  We write $R' =  (R \setminus X) \cup \{r_B \colon r \in X, B \in \mathcal{B}\}$, and we define \[\mathcal{S} = \{S_r \colon r \in R' \setminus (R_u \cup R_v \cup \{r_w\})\} \cup \{S_u, S_v, S_w\}.\]
    We observe that $\mathcal S$ is a subgraph partition of $S$.
    By our previous observations, each part of $\mathcal S$ intersects $A$ in a path of at most four vertices. Hence, for each part $S' \in \mathcal S$, each vertex pair in $S'$ has a distance in $A$ of at most $3$.

    For each $v \in V(S)$, $\ell_{\spl}(v) \geq \deg_S(v)$, and for each vertex $r' \in V(S)$ which is either a $3$-vertex in $G$ or is of the form $r_B$ for some $3$-vertex $r \in V(H)$ and $B \in \mathcal B$, $\ell_{\spl}(r') \geq \deg_S(r')  + 1$.
    Furthermore, the neighbors of $S_w$ in $A$ have distance at least $4$ in $A \setminus S_w$ and hence belong to distinct parts of $\mathcal S$.
    Therefore, by Lemma \ref{lem:weak-parts} (\ref{item:weak-petal}),
    $S_w$ is an $\ell_{\spl}$-weak part of $\mathcal S$.
    By Lemma \ref{lem:strong-parts} (\ref{item:strong-end-piece}), $S_u$ and $S_v$ are $\ell_{\spl}$-strong parts of $\mathcal S$.
    For each $r \in R' \setminus (R_u \cup R_v \cup \{r_w\})$ with a neighbor in $A$
    and part $S_r \in \mathcal S$, the neighbors of $S_r$ in $A$ have a distance of at least $5$ in $A \setminus S_r$ and hence belong to distinct parts of $\mathcal S$; therefore,
    by Lemma \ref{lem:strong-parts} (\ref{item:internal-petal}) 
    $S_r$ is an $\ell_{\spl}$-strong part of $\mathcal S$.
    Finally, for each  $r' \in R'$ for which $r'$ has no neighbor in $A$, $r'$ must be a vertex of the form $r_B$ for some $r \in X$ and $B \in \mathcal{B}$. 
    Therefore,
    $r$ has a neighbor in $A$ which is not a neighbor of $r'$,
    so
    Lemma \ref{lem:hanging-red} implies that $S_{r'} = r'$ is a strong part of $\mathcal S$.

    By applying Lemma \ref{lem:orig_partition}
    with $b=5$ and $d=4$, $S$ is $(\ell_{\spl},4,20^{-4}4^{-10})$-reductive. Hence, Lemma \ref{lem:crossing-over} implies that $H$ is $(4, 20^{-4}4^{-10}36^{-1})$-reducible,
    giving us a contradiction and completing the proof.
\end{proof}

\begin{lemma}
\label{lem:d-t}
    If $v \in V(G)$ is a $6^+$-vertex with exactly $t$ neighbors of degree $3$, then $v$ is conductively connected with at most $\deg(v) - t$ stressed vertices.
\end{lemma}
\begin{proof}
Write $d = \deg(v)$ and suppose that $v$ is conductively connected with $d-t+1$ stressed vertices. Call these stressed vertices $s_1, \dots, s_{d-t+1}$, and for each $i \in [d-t+1]$, let $P_i$ be a conductive path joining $v$ and $s_i$. Note that each path $P_i$ contains a $4$-neighbor $u_i \in N(v)$.

As $v$ has at most $d-t$ $4$-neighbors, the pigeonhole principle tells us that there exists a $4$-neighbor $u \in N(v)$ that belongs to two paths $P_i$ and $P_j$. If $u$ is stressed, then $u$ belongs to only one conductive path $P_i = (v,u)$; therefore, $u$ is conductive. Then, both $s_i$ and $s_j$ are conductively connected with the conductive vertex $u$, and hence $s_i$ and $s_j$ are conductively connected with each other, contradicting Lemma \ref{lem:con-stressed}. Therefore, $v$ is conductively connected with at most $d-t$ stressed vertices.
\end{proof}

\subsection{A global condition for $5$-vertices in a minimal counterexample}
This subsection is dedicated to proving the following lemma.
\begin{lemma}
    \label{lem:4-t} 
        If $x \in V(G)$ is a $5$-vertex with exactly $t$ neighbors of degree $3$, then $x$ is conductively connected with at most $4 - t$ stressed vertices.
    \end{lemma}

    Suppose for contradiction that $x$ is conductively connected with $s = 5 - t$ stressed vertices. We name
    these stressed vertices $y_1, \ldots, y_s$. For $i \in \{1, \ldots, s\}$, let $P_i$ be the shortest conductive path connecting $x$ and $y_i$. Observe that for all distinct $i, j \in \{1, \ldots, s\}$, $V(P_i) \cap V(P_j) = \{x\}$; otherwise the stressed vertices $y_i$ and $y_j$ are conductively connected, which contradicts Lemma \ref{lem:con-stressed}. Let $T = \bigcup_{i = 1}^s P_i$, and let $R$ be the set of $3$-neighbors of $T$ in $G$. We will demonstrate that $H = G[V(T) \cup R]$ is $(4, \xi)$-reducible. 

    First, we define $X = N_R(\{x, y_1, \ldots, y_s\})$ and prove the following claim. 
    \begin{claim}
    \label{claim:common-path}
        For each $r \in R \setminus X$, there exists $i \in \{1, \ldots, s\}$ such that $N_T(r) \subseteq V(P_i) -x$. 
    \end{claim}
    \begin{proof}
        Fix $r \in R \setminus X$. As $r \notin X$,
        $r \notin N_H(x)$ and $r \notin N_H(y_i)$ for all $i \in \{1, \ldots, s\}$. Thus, $N_T(r)$ contains conductive vertices only. Suppose for contradiction that for each $i \in \{1, \ldots, s\}$, $N_T(r) \not \subseteq V(P_i) -x$. As $N_T(r)$ is non-empty, $r$ is adjacent to a conductive vertex $p_i \in V(P_i) -x$ for some $i \in \{1, \ldots, s\}$, 
        and because $N_T(r) \not\subseteq V(P_i) -x$, $r$ is adjacent to another conductive vertex $p_j \in V(P_j) -x$ where $j \in \{1, \ldots, s\} \setminus \{i\}$. Let the conductive subpath from $y_i$ to $p_i$ be denoted $P_{y_i p_i}$, and let the conductive subpath from $y_j$ to $p_j$ be denoted $P_{y_j p_j}$. Then, $(P_{y_i p_i}, r, P_{y_j p_j})$ is a path of conductive edges joining $y_i$ and $y_j$ whose internal vertices consist of conductive vertices and one $3$-vertex $r \notin N_H(y_i) \cup N_H(y_j)$, contradicting Lemma \ref{lem:434-path}. Therefore, the claim holds.
    \end{proof}

    At this point, we aim to define a subgraph partition $\mathcal{B}$ of $H \setminus X$ in order to construct a split graph and apply Lemma \ref{lem:crossing-over}. Towards this end, for each $i \in \{1, \ldots, s\}$, define the vertex set
    \[
        A_i = \left(V(P_i) -x \right) \cup \left \{ r \in R \setminus X : N_T(r) \subseteq V(P_i) -x \right \} .
    \]
    In other words, $A_i$ is the set of vertices $r \in R$ whose neighborhoods in $T$ are contained in $V(P_i) - x$, along with $V(P_i) - x$. Observe that $\{x\} \cup \bigcup_{i = 1}^s A_i$ partitions $V(T) \cup (R \setminus X) = V(H) \setminus X$. Let $B_x = x$ and $B_i = H[A_i]$ for each $i \in \{1, \ldots, s\}$. Then, 
    $\mathcal{B} = \{B_x, B_1, \ldots, B_s\}$ is the desired subgraph partition of $H \setminus X$. Additionally, observe that for each $B \in \{B_1, \ldots, B_s\}$, $x$ has exactly one $4$-neighbor in $B$ and no $3$-neighbor in $B$.

    Now, we check the conditions for Lemma \ref{lem:crossing-over}. Observe that $s \leq 5$, $|N_R(y_i)| = 2$ for each $i \in \{1, \ldots, s\}$, and $|N_R(x)| \leq 3$ by Lemma \ref{lem:d-2}. Hence, $|X| = |N_R(\{x, y_1, \ldots, y_s\})| \leq |N_R(x)| + \sum_{i = 1}^s |N_R(y_i)| \leq 13$ and $|\mathcal{B}| = s + 1 \leq 6$, so condition (\ref{item:X-leq-q-and-B-leq-m}) is satisfied with $q = 13$ and $m = 6$. By the definition of $R$, for each $r \in R$, $r$ has a neighbor in $V(H) \setminus R = V(T)$, so condition (\ref{item:r-neighbor}) is satisfied.

    Next, we verify that condition (\ref{item:color-branch}) holds. For each $i \in \{1, \ldots, s\}$, let $H_i = H[V(B_i) \cup X]$. Fix $r \in X$ and $B \in \mathcal{B}$. We aim to show that $H \setminus (V(B) \cup \{r\})$ is $\ell_{H \setminus (V(B) \cup \{r\})}$-choosable. For notational simplicity, let $D = V(B) \cup \{r\}$. First, we fix $i \in \{1, \ldots, s\}$ and prove that if $B_i \neq B$, $B_i$ is $\ell_{H_i \setminus D}$-choosable. As $y_i$ is adjacent to distinct $r_1, r_2 \in X$, without loss of generality, we let $r_1 \in N_X(y_i)$ be a vertex distinct from $r$. 
    Since $r_1 \in X$ and $V(B) \subseteq V(H) \setminus X$, $r_1 \notin V(B)$. Hence, as $r_1 \notin V(B) \cup \{r\} = V(D)$, $r_1$ and $y_i$ are neighbors in $H_i \setminus D$. Because $B_i$ consists of $4^-$-vertices, for each $v \in V(B_i)$, $\ell_{H_i \setminus D}(v) \geq \deg_{H_i \setminus D}(v) \geq \deg_{B_i \setminus D}(v) = \deg_{B_i}(v)$, where the last equality comes from the fact that $V(B_i) \cap V(D) = \emptyset$ so $B_i \setminus D = B_i$. Additionally, we observe that $\ell_{H_i \setminus D}(y_i) \geq \deg_{H_i \setminus D}(y_i) \geq \deg_{B_i \setminus D}(y_i) + 1 = \deg_{B_i}(y_i) + 1$, where the inequality $\deg_{H_i \setminus D}(y_i) \geq \deg_{B_i \setminus D}(y_i) + 1$ is implied by the fact that $r_1$ is adjacent to $y_i$ in $H_i \setminus D$. Therefore, by Lemma \ref{lem:ERT}, $B_i$ is $\ell_{H_i \setminus D}$-choosable. 
    
    Next, we prove that if $B_x \neq B$, then $B_x$ is $\ell_{H \setminus D}$-choosable. Since $V(B_x) = \{x\}$, it suffices to show that $\ell_{H \setminus D}(x) \geq 1$. Observe that $\ell_{H \setminus D}(x) = 4 - |N_H(x) \cap D| = 4 - |N_H(x) \cap V(B)| - |N_H(x) \cap \{r\}|$. Because $x$ has exactly one neighbor in $V(B)$ and at most one neighbor in $\{r\}$, $\ell_{H \setminus D}(x) \geq 2$. Thus, $B_x$ is $\ell_{H \setminus D}$-choosable. Now, fixing $r \in X$, $B \in \mathcal{B}$, and an $\ell_{H \setminus D}$-assignment $L$, we may produce an $L$-coloring of $H \setminus D$ by first coloring $B_x$ if $B_x \neq B$, then coloring $B_i$ for each $i \in \{1, \ldots, s\}$ where $B_i \neq B$, and finally coloring $X \setminus \{r\}$, using the fact that $\ell_{H[X]}(r') \geq \deg_{H[X]}(r') + 1$ for each $r' \in X$ since $X$ contains only $3$-vertices. This verifies that condition (\ref{item:color-branch}) holds.

    Now, we show that condition (\ref{item:r1r2}) is satisfied. Fix $r_1, r_2 \in X$ and an $\ell_{H \setminus \{r_1, r_2\}}$ assignment $L$ on $H \setminus \{r_1, r_2\}$. We aim to show that $H \setminus \{r_1,r_2\}$ is $L$-colorable. Recall that $\ell_{H \setminus \{r_1, r_2\}}(v) = \deg_{H \setminus \{r_1, r_2\}}(v)$ for each $4$-vertex in $v \in V(H)
    \setminus \{r_1, r_2\}$, and $\ell_{H \setminus \{r_1, r_2\}}(v) = \deg_{H \setminus \{r_1, r_2\}}(v) + 1$ for each $3$-vertex in $v \in V(H) \setminus \{r_1, r_2\}$. 
    Let 
    \[
        \mathcal{Q} = \{B_i: B_i \in \{B_1, \ldots, B_s\} \textrm{ and } N_R(y_i) = \{r_1, r_2\} \}.
    \]
    To obtain an $L$-coloring of $H \setminus \{r_1, r_2\}$, we will first color each $B_i \in \mathcal{Q}$, then color $x$, then color each $B_i \notin \mathcal{Q}$, and finally color the remaining uncolored vertices in $X$. 
    Towards this end, we observe that for each $i \in \{1, \ldots, s\}$, if $B_i \notin \mathcal{Q}$, then $N_R(y_i) \neq \{r_1, r_2\}$, which means that $y_i$ is adjacent to a $3$-vertex $r \in X \setminus \{r_1, r_2\}$. 
    Additionally, we consider the quantity 
    \[
        \sigma = |N_H(x) \cap \{r_1, r_2\}| + \left\vert N_H(x) \cap \left(\bigcup_{B_i \in \mathcal{Q}} V(B_i)\right) \right\vert,
    \]
    which we use to lower bound the number of colors available in $L(x)$ after all $B_i \in \mathcal{Q}$ have been colored. If $\sigma \leq 3$, then we may obtain an $L$-coloring of $H \setminus \{r_1, r_2\}$ as follows. First, color $V(B_i)$ for each $B_i \in \mathcal{Q}$. Because each $B_i \in \mathcal{Q}$ contains only $4^-$-vertices and is adjacent to $x$, which is uncolored, coloring $\bigcup_{B_i \in \mathcal{Q}} V(B_i)$ can be accomplished using Lemma \ref{lem:ERT}. Second, color $x$. This is possible because at the time $x$ is colored, the number of available colors at $x$ is at least $4 - |N_H(x) \cap \{r_1, r_2\}| - |N_H(x) \cap \bigcup_{B_i \in \mathcal{Q}} V(B_i)| = 4 - \sigma \geq 1$. Third, color $V(B_i)$ for all $B_i \notin \mathcal{Q}$. Applying Lemma \ref{lem:ERT}, this is possible because each $B_i \notin \mathcal{Q}$ contains only $4^-$-vertices and is adjacent to an uncolored $r \in X \setminus \{r_1, r_2\}$. Finally, color all $r \in X \setminus \{r_1, r_2\}$.

    To prove finish proving (\ref{item:r1r2}), we show that $\sigma \leq 3$. First, consider the case where $|\mathcal{Q}| \geq 2$. Then there exist distinct $B_i, B_j \in \mathcal{Q}$. Hence, $N_R(y_i) = N_R(y_j) = \{r_1, r_2\}$. Lemma \ref{lem:stressed-special-K4} implies that $H[\{y_i, y_j, r_1, r_2\}]$ is a special $K_4$. Hence, $r_1$ and $r_2$ are adjacent, which means that $ N_H(x) \cap \{r_1, r_2\}  = \emptyset$ and $|\mathcal{Q}| = 2$. Then $\sigma = |N_H(x) \cap \{r_1, r_2\}| + |N_H(x) \cap (\bigcup_{B_i \in \mathcal{Q}} V(B_i))| = |N_H(x) \cap (\bigcup_{B_i \in \mathcal{Q}} V(B_i))| = 2$, using the fact that $|N_H(x) \cap V(B_i)| = 1$ for all $i \in \{1, \ldots, s\}$. On the other hand, if $|\mathcal{Q}| \leq 1$, then $\sigma = |N_H(x) \cap \{r_1, r_2\}| + |N_H(x) \cap (\bigcup_{B_i \in \mathcal{Q}} V(B_i))| \leq 2 + 1 = 3$. In both cases, $\sigma \leq 3$, demonstrating that condition (\ref{item:r1r2}) is satisfied.

    Finally, we verify condition (\ref{item:split-coloring}). Define $S = \spl(H, X, \mathcal{B})$. For the remainder of this proof, for each vertex $v \in V(H) \setminus X$, we identify the copy of $v$ in $S$ with $v$. Similarly, we identify the copy of $T$ in $S$ with $T$. Additionally, we let $R_S = (R \setminus X) \cup \left(\{r_B \colon r \in X, B \in \mathcal{B} \}\right)$. 

    \begin{lemma}
        $S$ is $(\ell_{\spl}, 4, \gamma)$-reductive for $\gamma = 25^{-4} 4^{-10}$. 
    \end{lemma}
    \begin{proof}
        To prove this claim, we partition $S$ into $\ell_{\spl}$-strong and $\ell_{\spl}$-weak parts. Let $R_0$ be the set of vertices $r \in R_S$ with no neighbor in $T$. Let $R_x$ be the set of vertices $r \in R_S$ for which $N_T(r) = \{x\}$. Finally, let $R_b = R_S \setminus (R_0 \cup R_x)$. Observe that $R_0, R_x$, and $R_b$ partition $R_S$. We give a more convenient characterization of $R_b$. By the construction of $S$, for each $r \in \{r_B \colon r \in X, B \in \mathcal{B}\}$, all neighbors of $r$ belong to the same part of $ \mathcal{B}$. Additionally, Claim \ref{claim:common-path} implies that for each $r \in R \setminus X$, $N_T(r) \subseteq V(P_i) -x \subseteq V(B_i)$ for some $i \in \{1, \ldots, s\}$. Hence, for all $r \in R_S$, all neighbors of $r$ in $T$ belong to the same part of $\mathcal{B}$. We conclude that for all $r \in R_S$, the neighborhood of $r$ in $T$ is one of the following: the empty set, $\{x\}$, or a non-empty subset of $V(P_i) -x$ for some $i \in \{1, \ldots, s\}$. We observe that $R_b = \{r \in R_S \colon \exists i \in \{1, \ldots, s\}, \emptyset \neq N_T(r) \subseteq V(P_i) -x\}$. In other words, $R_b$ is the set of vertices $r \in R_S$ such that the neighborhood of $r$ in $T$ is a nonempty subset of $V(P_i) - x$ for some $i \in \{1, \ldots, s\}$.

        Now, we define a subgraph partition $\mathcal{S}$ of $S$. First, define $\mathcal{S}_0 = \{r \colon r \in R_0\}$ (that is, $\mathcal S_0$ is a family of $K_1$ subgraphs of $S$). Next, for each $i \in \{1, \ldots, s\}$, let $R_{y_i} = N_{R_S}(y_i)$. Observe $R_{y_i} \subseteq R_b$. Let $S_{y_i} = S[N_T[R_{y_i}]]$. Then, for each $r \in R_b \setminus \bigcup_{i = 1}^s R_{y_i}$, let $S_r = S[N_T[r]]$. We observe that for $r \in  R_b \setminus \bigcup_{i = 1}^s R_{y_i}$, $N_T(r)$ consists of conductive vertices. Next, define $\mathcal{S}_b = \{S_{y_i} \colon i \in \{1, \ldots, s\} \} \cup \{S_r \colon r \in  R_b \setminus \bigcup_{i = 1}^s R_{y_i}\}$. Finally, let $S_x = S[N_{R_x}[x]]$, let $\mathcal{S}_x = \{S_x\}$, and let $\mathcal{S} = \mathcal{S}_0 \cup \mathcal{S}_b \cup \mathcal{S}_x$. We make the observation that for each $i \in \{1, \ldots, s\}$, the only part to which $y_i$ belongs is $S_{y_i}$. 

        We show $\mathcal{S}$ forms a subgraph partition of $S$. Observe that for $\lambda \in \{0, x, b\}$ and $r \in R_{\lambda}$, $r$ appears in exactly one part in $\mathcal{S}_{\lambda}$, and $r$ does not appear in any part in $\mathcal{S} \setminus \{\mathcal{S}_{\lambda}\}$. Hence, for all $r \in R_S$, $r$ appears in exactly one part of $\mathcal{S}$. Now, we show that for each $v \in V(T)$, $v$ appears in exactly one part of $\mathcal{S}$. 
        If $v = x$, then $v$ belongs to exactly one part in $\mathcal{S}$, namely $S_x$. Otherwise, if $v \in V(T) -  x$, then $v$ has a neighbor $r \in R_b$. Because $r \in V(S_u)$ for some $S_u \in \mathcal{S}_b$, we see that $v \in V(S_u)$, using the fact that $S_u$ is defined to contain $N_T(r)$. Hence, $v$ belongs to at least one part of $\mathcal{S}$. Now, suppose for contradiction that $v$ belongs to two parts
        of $\mathcal{S}$. Since $v$ belongs to parts in only $\mathcal{S}_b$, there exist distinct $S_{u_1}, S_{u_2} \in \mathcal{S}_b$ such that $v \in V(S_{u_1})$ and $v \in V(S_{u_2})$. Then,
        there exist $r_1 \in R_S \cap V(S_{u_1})$ and $r_2 \in R_S \cap V(S_{u_2})$ such that $v$ is adjacent to both $r_1$ and $r_2$. Since each $r \in R_S$ appears in exactly one part in $\mathcal{S}$, $r_1 \neq r_2$. Hence, $v$ is stressed, and $v = y_i$ for some $i \in \{1, \ldots, s\}$. We conclude that $y_i$ belongs to two distinct parts $S_{u_1}$ and $S_{u_2}$, which contradicts the above observation. Hence we conclude that $v$ belongs to exactly one part in $\mathcal{S}$. Since all vertices in $V(S)$ belong to exactly one part in $\mathcal{S}$, $\mathcal{S}$ partitions $V(S)$, so $\mathcal{S}$ forms a subgraph partition of $S$.

        Using Lemma \ref{lem:hanging-red}, we see that for each $r \in \mathcal{S}_0$, $r$ is an $\ell_{\spl}$-strong part. 

        Next, we make the following observation and claims about $\mathcal S_b$.
        \begin{observation}
        \label{obs:list_size_in_S}
            For each $v \in V(H) \setminus X$, $\ell_{\spl}(v) = \ell_H(v)$. Additionally, for each $r_B \in \{r_B \colon r \in X, B \in \mathcal{B}\}$, $\ell_{\spl}(r_B) = \ell_H(r) = \deg_H(r) + 1 \geq \deg_S(r_B) + 1$. 
        \end{observation}

        \begin{claim}
            \label{claim:form_path}
            For each $r \in R_b$, $T[N_T(r)]$ forms a subpath of $P_i -x$ for some $i \in \{1, \ldots, s\}$.
        \end{claim}
        \begin{proof}
            Fix $r \in R_b$. Since $r \in R_b$, $N_T(r) \neq \emptyset$ and $N_T(r) \subseteq V(P_i) -x$ for some $i \in \{1, \ldots, s\}$. The claim is clearly true when $|N_T(r)| = 1$, so suppose that $|N_T(r)| \geq 2$. Let $u, v \in V(P_i) -x$ be the neighbors of $r$ that are farthest apart in $P_i -x$. Observe that $u$ and $v$ are joined by a conductive induced subpath of $P_i -x$, which we denote $P_{uv}$. By Lemma \ref{lem:far_neighbors}, every vertex $p \in V(P_{uv})$ is adjacent to $r$. Hence, $T[N_T(r)]$ forms a subpath of $P_i -x$. 
        \end{proof}

        \begin{claim}
            \label{claim:contain_H_T_intersection}
            For each $S' \in \mathcal{S}_b$, there exists $r' \in R_b$ such that $V(S') \cap V(T) = N_T(r')$.
        \end{claim}
        \begin{proof}
            If $S' = S_{r'}$ for some $r' \in R_b \setminus \bigcup_{i = 1}^s R_{y_i}$, then the claim is true by the definition of $S_{r'}$. So, suppose that $S' = S_{y_i}$ for $i \in \{1, \ldots, s\}$. Let $R_{y_i} = \{r_1, r_2\}$. Since $r_1, r_2 \in R_b$, we have $N_T(r_1), N_T(r_2) \subseteq V(P_i) -x$. 
            Then, choose $p \in N_T(R_{y_i}) = N_T(r_1) \cup N_T(r_2) \subseteq V(P_i) -x$ such that the distance from $p$ to $y_i$ in $P_i -x$ is maximized. Let the subpath of $P_i -x$ be called $P_{y_i p}$. Without loss of generality, assume that $p \in N_T(r_1)$. Then, by Lemma \ref{lem:far_neighbors}, all vertices in $P_{y_i p}$ are neighbors of $r_1$. Hence, $V(S') \cap V(T) = N_T(R_{y_i}) \subseteq N_T(r_1)$. Because $N_T(r_1) \subseteq N_T(R_{y_i})$, we have $V(S') \cap V(T) = N_T(r_1)$, proving the claim.
        \end{proof}

        Now, we show that for all parts $S' \in \mathcal{S}_b$, $S'$ is an $\ell_{\spl}$-strong part. First, we demonstrate that for each $i \in \{1, \ldots, s\}$, $S_{y_i}$ is a strong part. Observe that $R_{y_i} \subseteq R_b$ and that $|R_{y_i}| = 2$. Let $R_{y_i} = \{r_1, r_2\}$. Applying Claim \ref{claim:contain_H_T_intersection}, we may assume without loss of generality that 
        $N_T(R_{y_i}) = N_T(r_1)$. By using Claim \ref{claim:form_path}, we see that $T[N_T(R_{y_i})] = T[N_T(r_1)]$ is a subpath of $P_i -x$. We call this subpath $P_{r_1} = (v_1, \ldots, v_t)$ where $v_1 = y_i$ and $1 \leq t \leq 3$. Now, we demonstrate that $N_T(r_2) = \{y_i\}$. Recall from the definition of $R_{y_i}$ that $\{y_i\} \subseteq N_T(r_2)$. Next, if $v \in N_T(r_2)$, then $v \in N_T(R_{y_i}) = N_T(r_1)$. Hence, $v$ is stressed. As $r_2 \in R_b$, $N_T(r_2) \subseteq V(P_i) -x$. Thus, because $y_i$ is the only stressed vertex in $V(P_i) -x$, we conclude that $v = y_i$, proving that $N_T(r_2) = \{y_i\}$. Hence, $S_{y_i}$ is formed from a petal graph with root $r_1$ and path $P_{r_1}$ by adding $r_2 \in N_S(y_i)$ and possibly the edge $r_1 r_2$; no other edges with endpoint $r_2$ are added to the petal graph. Next, by applying Observation \ref{obs:list_size_in_S}, we see that $\ell_{\spl}(r_1) \geq \deg_S(r_1) + 1$, $\ell_{\spl}(r_2) \geq \deg_S(r_2) + 1$, and for each $v \in P_{r_1}$, $\ell_{\spl}(v) = \ell_H(v) = \deg_H(v) = \deg_S(v)$. Finally, we observe that $v_t \in P_{r_1}$ has a neighbor in $P_i$ from a part other than $S_{y_i}$, completing the proof that $S_{y_i}$ is an $\ell_{\spl}$-strong part.

        Next, we argue that $S_r$ is an $\ell_{\spl}$-strong part for each $r \in R_b \setminus \bigcup_{i = 1}^s R_{y_i}$. By Claim \ref{claim:form_path}, there exists $i \in \{1, \ldots, s\}$ such that $T[N_T(r)]$ is a subpath of $P_i -x$, which we call $P_r = (v_1, \ldots, v_t)$.
        We note that $t \leq \deg_S(r) \leq 3$.
        Therefore, $S_r$ is a petal graph with a root $r$ and path $P_r$ on at most three vertices. For each $v \in N_T(r)$, we apply Observation \ref{obs:list_size_in_S} to obtain $\ell_{\spl}(v) = \ell_H(v) = \deg_H(v) = \deg_S(v)$. Now, if $r \in R \setminus X$, Observation \ref{obs:list_size_in_S} gives $\ell_{\spl}(r) = \ell_H(r) = \deg_H(r) + 1 \geq \deg_S(r) + 1$. On the other hand, if $r \in \{r_B \colon r \in X, B \in \mathcal{B}\}$, then $\ell_{\spl}(r) \geq \deg_S(r) + 1$ by Observation \ref{obs:list_size_in_S}. Let $w$ be the neighbor of $v_1$ in $T$ and $z$ be the neighbor of $v_t$ in $T$. Observe that $w$ and $z$ are distinct. We demonstrate that $w$ and $z$ belong to distinct parts in $\mathcal{S}$. Suppose for contradiction that for some $S' \in \mathcal{S}$, $w, z \in V(S')$. If $x \in \{w,z\}$, then $S' = S_x$. Because $V(S_x) \cap V(T) = \{x\}$, $w = z = x$, which contradicts the fact that $w$ and $z$ are distinct. If $x \not \in \{w,z\}$, then $S' \in \mathcal{S}_b$. By Claim \ref{claim:contain_H_T_intersection}, $w, z \in N_T(r')$ for some $r' \in R_b$. But then $N_T(r')$ is not connected in $T -x$, which contradicts Claim \ref{claim:form_path}. Hence, the vertices $w, z \in V(T)$ that are adjacent to $S_r$ belong to different parts in $\mathcal{S}$. We conclude that $S_r$ is an $\ell_{\spl}$-strong part.

        Finally, we demonstrate that $S_x \in \mathcal{S}$ is an $\ell_{\spl}$-weak part. We observe that $S_x$ can be obtained from a star $K_{1, t}$ with center $x$ and leaves $R_x$ by adding edges between distinct vertices in $R_x$. Additionally, by Lemma \ref{lem:d-2}, $t = |R_x| \leq 3$. From Observation \ref{obs:list_size_in_S}, we see that $\ell_{\spl}(x) = \ell_H(x) = 4 = \deg_H(x) - 1 = \deg_S(x) - 1$ and that for each $r \in R_x$, $\ell_{\spl}(r) \geq \deg_S(r) + 1$. Now, we show that given distinct $u, v \in N_T(x)$, $u$ and $v$ belong to distinct parts in $\mathcal{S}$. 
        Suppose for contradiction that $u, v \in S'$ for some $S' \in \mathcal{S}$. Observe that $S' \in \mathcal{S}_b$. Then, by Claim \ref{claim:contain_H_T_intersection}, there exists $r \in R_b$ such that $u, v \in N_T(r) = V(S') \cap V(T)$. As $\{u, v\} \not \subseteq V(P_i) -x$ for each $i \in \{1, \ldots, s\}$, we conclude that $N_T(r) \not \subseteq V(P_i) -x$ for each $i \in \{1, \ldots, s\}$, which contradicts our characterization of $R_b$. As each $u \in N_T(x)$ belongs to a part in $\mathcal{S} - S_x$, we conclude that $x$ has a neighbor in at least $|N_T(x)|$ parts in $\mathcal{S} - S_x$. As $x$ is conductively connected to $s = 5 - t$ stressed vertices through conductive paths whose pairwise intersections contain only $x$, $|N_T(x)| = 5 - t$, proving that $x$ has a neighbor in at least $5 - t$ parts in $\mathcal{S} - S_x$. Hence, by Lemma \ref{lem:deg-5-weak-part}, $S_x$ is an $\ell_{\spl}$-weak part.

        Using $b = 5$ and $d = 5$, we apply Lemma \ref{lem:orig_partition} to prove that $S$ is $(\ell_{\spl}, 4, 25^{-4} 4^{-10})$-reductive. Therefore, condition (\ref{item:split-coloring}) of Lemma \ref{lem:crossing-over} is satisfied.
    \end{proof}
    
    Because conditions (\ref{item:X-leq-q-and-B-leq-m})-(\ref{item:split-coloring}) are satisfied, we conclude from Lemma \ref{lem:crossing-over} that $H$ is $(4, \alpha)$-reducible for the constant $\alpha = \min\{\frac{25^{-4}4^{-10}}{3 \cdot 13 \cdot 6}, \frac{1}{12 \cdot 13^2}\} = \frac{25^{-4}4^{-10}}{234}$.
    As $\xi < \alpha$, 
    we reach a contradiction, and the proof of Lemma \ref{lem:4-t} is complete.

\section{Discharging}
\label{sec:discharging}
In this section, we finish the proof of Theorem \ref{thm:mad113} by showing
that our minimal counterexample $G$ has maximum average degree at least $\frac{11}{3}$, a contradiction.
We apply the following discharging rules on $G$:
\begin{enumerate}
    \item[(R1)] Each vertex $v$ receives a charge equal to $\deg(v) - \frac{11}{3}$.
    \item[(R2)] Each vertex of degree $3$ takes a charge of $\frac{1}{3}$ from each $4^+$-neighbor.
    \item[(R3)] Each stressed vertex $v$ takes a charge of $\frac{1}{3}$ from an insulated vertex with which $v$ is conductively connected.
\end{enumerate}

As the maximum average degree of $G$ is less than $\frac{11}{3}$, the total charge assigned to the vertices of $G$ by the rule (R1) is negative. Furthermore, none of our discharging rules creates or destroys charge. Therefore, after applying (R1), (R2), (R3), the sum of the charges of the vertices in $G$ is negative. 

Now, we aim to show that for each vertex $v \in V(G)$, the charge of $v$ after applying (R1), (R2), (R3) is nonnegative.
This will imply that the sum of the charges of the vertices in $G$ is nonnegative after applying (R1), (R2), (R3), giving us a contradiction.

\begin{itemize}
    \item If $\deg(v) = 3$, then by Lemma \ref{lem:d-2}, $v$ has at least two $4^+$-neighbors. Therefore, the final charge of $v$ is at least $3 - \frac{11}{3} + 2 \cdot \frac{1}{3} = 0$.
    \item If $\deg(v) = 4$ and $v$ is conductive, then as $v$ has exactly
    one neighbor of degree $3$, the final charge of $v$ is at least $4 - \frac{11}{3} - \frac{1}{3} = 0$.
    \item If $\deg(v) = 4$ and $v$ is stressed, then $v$ has exactly two neighbors of degree $3$.  Hence, after applying (R1) and (R2), $v$ has charge $4 - \frac{11}{3} - 2 \cdot \frac{1}{3} = -\frac{1}{3}$. Then, by Lemma \ref{lem:stressed-insulated}, $v$ is conductively connected with an insulated vertex. Therefore, after applying (R3), $v$ has a final charge of $-\frac{1}{3} + \frac{1}{3} = 0$.
    \item If $\deg(v) = 4$ and $v$ is insulated, then after applying (R1) and (R2), the charge of $v$ is $4 - \frac{11}{3} = \frac{1}{3}$. Then, by Lemma \ref{lem:at-most-one-stressed}, $v$ loses at most $\frac{1}{3}$ charge while applying (R3), and hence the final charge of $v$ is at least $0$.
    \item If $\deg(v) = 5$ and $v$ has $t$ neighbors of degree $3$, 
    then after applying (R1) and (R2), $v$ has a charge of $5 - \frac{11}{3} - \frac{1}{3}t = \frac{4}{3} - \frac{1}{3}t $. Furthermore, by Lemma \ref{lem:d-2}, $t  \leq 3$, and by Lemma \ref{lem:4-t}, $v$ is conductively connected with at most $4-t$ stressed vertices. Therefore, when applying (R3), $v$ loses at most $\frac{1}{3}(4-t)$ charge and hence has a final charge of at least $\frac{4}{3} - \frac{1}{3}t - \frac{1}{3}(4-t) = 0$.
    \item If $\deg(v) \geq 6$ and $v$ has $t$ neighbors of degree $3$, then after applying (R1) and (R2), $v$ has a charge of $\deg(v) - \frac{11}{3} - \frac{1}{3} t$. Then, by Lemma \ref{lem:d-t}, $v$ is conductively connected with at most $\deg(v) - t$ stressed vertices; therefore, after applying (R3), $v$ has a charge of at least $\deg(v) - \frac{11}{3} -\frac{1}{3}t - \frac{1}{3}(\deg(v) - t) = \frac{2}{3}\deg(v) - \frac{11}{3} > 0$.
\end{itemize}
As a $4$-vertex has at most two $3$-neighbors by Lemma \ref{lem:d-2}, every $4$-vertex of $G$ is either insulated, conductive, or stressed. Therefore, our cases are exhaustive, and hence the final charge of each $v \in V(G)$ is nonnegative, giving a contradiction.

From this argument, we see that every graph of maximum average degree less than $\frac{11}{3}$ has a $(4,\xi)$-reducible subgraph. 
In particular, our minimal counterexample $G$ has a $(4,\xi)$-reducible subgraph, and as every subgraph of $G$ also has maximum average degree less than $\frac{11}{3}$, 
every induced subgraph of $G$ has a $(4,\xi)$-reducible subgraph. Therefore, Lemma \ref{lem:main-k-red} implies that $G$ is weighted $\epsilon$-flexibly $4$-choosable for $\epsilon = \frac{2\xi^3}{4} = 2^{-145}$,
and therefore $G$ in fact is not a counterexample.
This completes the proof of Theorem \ref{thm:mad113}.

\section{Further directions}
It is likely that the value $\frac{11}{3}$ in Theorem \ref{thm:mad113} is not optimal and that more involved structural lemmas would lead to a discharging argument with a 
stronger conclusion. 
Given that many arguments in the paper are already highly technical, we have opted to avoid adding further technicalities which may obscure the application of our tools. 

\bibliographystyle{plain}
\bibliography{113bib}

\section{Appendix}
\begin{claim}
    In Lemma \ref{lem:far_neighbors}, for each induced Gallai tree subgraph $T$ of $H$, $\ell_H$ is not bad on $T$.
\end{claim}
\begin{proof}
Suppose the claim is false, and let $T$ be a counterexample.     
    If $T$ is a single vertex $w$, then $\ell_H(w) \geq 2 = \deg_T(w) + 2$. 
    If $T$ is a $K_2$, then some $w \in V(T)$ satisfies $\ell_H(w) \geq 3 = \deg_T(w) + 2$.
    Therefore, $|V(T)| \geq 3$.

    Suppose $T$ contains $x$.  If $u \notin V(T)$ or $v \notin V(T)$, then $\ell_H(x) = \deg_H(x) + 1 \geq \deg_T(x) + 2$, a contradiction; therefore, $u,v \in V(T)$. If some $p \in V(P)$ does not belong to $T$, then two vertices $y,z \in V(T) \cap V(P)$ have a neighbor in $P$ which is not a neighbor in $T$; then, three vertices $w \in V(T)$ satisfy $\ell_H(w) \geq \deg_T(w) + 1$, a contradiction. Therefore, $V(P) \subseteq V(T)$.
    As $H[V(P) \cup \{x\}]$ is a cycle in $T$, it follows that $|V(P)| \geq 4$.
    Therefore, if $|R| = 1$, then the single vertex in $R$ has at least two neighbors in $P$ and thus does not belong to $T$, as otherwise $T$ is not a Gallai tree. Otherwise, $|R| \geq 2$, and we consider two distinct $r_1, r_2 \in R$.
    For each $r_i$, either $r_i \in V(T)$, or $r_i \not \in V(T)$ but some neighbor of $r_i$ belongs to $V(T)$. Then, $T$ again has three vertices $w$ for which $\ell_H(w) \geq \deg_T(w) + 1$, a contradiction.

    On the other hand, suppose that $T$ does not contain $x$. Let $P_T \subseteq P$ be a minimal subpath of $P$ containing $V(T) \cap V(P)$. 
    If $|V(P_T)| = 1$, then as we have already ruled out the case that $T$ is a single vertex or a $K_2$, 
    $T$ has at least three vertices,
    and each vertex $w \in V(T)$ satisfies $\ell_H(w) \geq \deg_T(w) + 1$. 
    Therefore, $|V(P_T)| \geq 2$.
    Note that each endpoint $w$ of $P_T$ has a neighbor in $(P \cup \{x\}) \setminus V(T)$; 
    therefore, $\ell_H(w) \geq \deg_T(w) + 1$. If $V(T) \cap V(P) \subsetneq V(P_T)$, then either an
    endpoint $w $ of $P_T$ satisfies $\ell_H(w) \geq \deg_T(w) + 2$, or three vertices $w \in V(T)$ satisfy $\ell_H(w) \geq \deg_T(w) + 1$, a contradiction. If $V(T) \cap V(P) = V(P_T)$, then consider a vertex $r \in N_R(P_T)$
    (note that $N_R(P_T) \neq \emptyset$, as otherwise $V(P_T) = V(P) = \{u,v\}$, a contradiction). If $r \in V(T)$, then $r$ is a third vertex of $T$ satisfying $\ell_H(r) \geq \deg_T(r) + 1$. Otherwise, let $z \in V(P_T)$ be a neighbor of $r$. If $z$ is an endpoint of $P_T$, then $\ell_H(z) \geq \deg_T(z) + 2$; otherwise, $z$ is a third vertex of $T$ satisfying $\ell_H(z) \geq \deg_T(z) + 1$. Hence, $T$ is in fact not a counterexample.
\end{proof}
\begin{claim}
    If $|R| \geq 2$ in the proof of Lemma \ref{lem:cycle}, then for each induced Gallai tree subgraph $T$ of $H$, $\ell_H$ is not bad on $T$.
\end{claim}
\begin{proof}
    Suppose that the claim does not hold, and let $T$ be a counterexample.
    If $T$ is a single vertex $w$, then $\ell_H(w) \geq 2 = \deg_T(w) + 2$, a contradiction.
    If $|V(T)| = 2$, then for some $w \in V(T)$, 
    $\ell_H(w) \geq 3 =  \deg_T(w) + 2$, a contradiction. If $|V(T)| \geq 3$ and $|V(T) \cap V(C)| \leq 1$, then each vertex $w \in V(T)$ satisfies $\ell_H(w) \geq \deg_T(w) + 1$, a contradiction.
    Otherwise, $|V(T) \cap V(C)| \geq 2$.
    If $V(T) \cap V(C) \subsetneq V(C)$, then $T$ has two vertices $w$ with a neighbor in $C$ that is not a neighbor in $T$; hence, each of these vertices satisfies $\ell_H(w) \geq \deg_T(w) + 1$. 
    If $T$ also has a vertex $r \in R$, then $\ell_H(r) \geq \deg_T(r) + 1$, and thus $T$ is not a counterexample. Therefore, some vertex $w \in V(T) \cap V(C)$ has a neighbor in $C$ which is not a neighbor in $T$, as well as a neighbor in $R$ which is not a neighbor in $T$. Therefore, $\ell_H(w) \geq \deg_T(w) + 2$, and $\ell_H$ is not bad on $T$.

     Hence, we may assume $V(C) \subseteq V(T)$. Let $a$ be the number of $3$-vertices in $V(T)$. Let $b$ be the number of vertices $w \in V(C)$ whose $3$-neighbor is not an element of $V(T)$. Consider a vertex $w \in V(C)$ whose $3$-neighbor $r$ is in $V(T)$. Because $|R| \geq 2$, $N_C(r) \subsetneq V(C)$, which means $w$ is the unique neighbor of $r$ in $C$; otherwise, the block containing $C$ and $r$ is neither a clique nor an odd cycle. Hence, $a \geq |V(C)| - b \geq 3 - b$. Since each $r \in R$ satisfies $\ell_H(r) \geq \deg_H(r) + 1 \geq \deg_T(r) + 1$, each $w \in V(C)$ whose $3$-neighbor is not in $T$ satisfies $\ell_H(w) \geq \deg_T(w) + 1$, and $a + b \geq 3$, $T$ has three vertices $w$ for which $\ell_H(w) \geq \deg_T(w) + 1$.
     Therefore, $T$ is not a counterexample to the claim, and the claim holds.
\end{proof}
\begin{claim}
    In the proof of Lemma \ref{lem:stressed-triangle}, if $T$ is an induced Gallai tree subgraph of $H$, then $\ell_H$ is not bad on $T$.
\end{claim}
\begin{proof}
    Suppose the claim is false, and let $T$ be a counterexample.
    If $T$ consists of a single vertex $w$, then $\ell_H(w) \geq 2 = \deg_T(w) + 2$. If $T$ is a $K_2$, then some $w \in V(T)$ satisfies $\ell_H(w) \geq 3 = \deg_T(w) + 2$. If $T$ is a triangle, then each $w \in V(T)$ satisifes $\ell_H(w) \geq 3 = \deg_T(w) + 1$. If $T$ is a $K_4$, then three of the vertices $w \in V(T)$ satisfy $\ell_H(w) = 4 = \deg_T(w) + 1$. Therefore, $T$ is not a clique.

    Write $r_1, r_2$ for the two $3$-neighbors of $v$, and suppose first that $r_1 r_2 \in E(H)$. Then, $\ell_H(w) \geq 3$ for each $w \in V(H)$, so no $w \in V(T)$ satisfies $\deg_T(w) = 1$. As $T$ is not a clique, it follows that $T$ has two triangle terminal blocks. Thus, $T$ has four vertices $w$ for which $\ell_H(w) \geq 3 \geq \deg_T(w) + 1$, a contradiction.

    Now, suppose that $r_1 r_2 \not \in E(H)$. Suppose $T$ has a triangle terminal block $B$. If  some $w \in \{v,r\}$ belongs to $B$ and is not a cut-vertex of $T$, then $\ell_H(w) = 4 = \deg_T(w) + 2$.
    Otherwise, $B$ contains some non cut-vertex $w \in N(v)$ for which $\ell_H(w) = \deg_H(w) = 4 = \deg_T(w) + 2$.    
    Therefore, the only remaining case is that $T$ is a tree with at least two edges. If $T$ has a leaf $w \in V(C) \cup \{r\}$, then $\ell_H(w) \geq 3 = \deg_T(w) + 2$. Otherwise, $T = H[N_R[v]]$, and hence all three vertices $w \in V(T)$ satisfy $\ell_H(w) \geq \deg_T(w) + 1$. Therefore, $T$ is not a counterexample, and the claim holds.
\end{proof}

\end{document}